\documentclass[11pt,leqno]{amsart}

\usepackage[refpage]{nomencl}
\usepackage{amssymb}
\usepackage{amsmath}
\usepackage{amsxtra}
\usepackage{amscd}
\usepackage{mathtools}
\usepackage{float}
\usepackage{graphicx}
\usepackage{amsfonts}
\usepackage{pb-diagram}
\usepackage[utf8]{inputenc}
\usepackage[T1]{fontenc}
\usepackage{enumerate}
\usepackage{array,multirow}
\usepackage{color}
\usepackage[all]{xy}
\usepackage{cases}
\usepackage{empheq}
\usepackage{changes}
\definechangesauthor[color=orange,name={Aristides Kontogeorgis}]{AK}
\usepackage{amsbsy}
\usepackage{soul}
\usepackage{cancel}
\usepackage{calc}
\usepackage{mathdots}
\usepackage{comment}
\usepackage[linenumbers]{MagmaTeX}
\newsavebox\CBox
\newcommand\hcancel[2][0.5pt]{%
  \ifmmode\sbox\CBox{$#2$}\else\sbox\CBox{#2}\fi%
  \makebox[0pt][l]{\usebox\CBox}%
  \rule[0.5\ht\CBox-#1/2]{\wd\CBox}{#1}}

\newtheorem{theorem}{Theorem}

\newtheorem{corollary}[theorem]{Corollary}
\newtheorem{lemma}[theorem]{Lemma}
\newtheorem{proposition}[theorem]{Proposition}
\theoremstyle{definition}
\newtheorem{definition}[theorem]{Definition}
\newtheorem{example}[theorem]{Example}
\theoremstyle{remark}

\newtheorem{remark}[theorem]{Remark}

\magmanames(depthFirstSearch,nodeVisitor,strongComponents,acyclicQuotient,transitiveQuotient)

\numberwithin{equation}{section}

\makeatletter
\let\@wraptoccontribs\wraptoccontribs

\renewcommand*\env@matrix[1][*\c@MaxMatrixCols c]{%
  \hskip -\arraycolsep
  \let\@ifnextchar\new@ifnextchar
  \array{#1}}
\makeatother

\begin{document}

\title{Biduality and Reflexivity in Positive Characteristic }

\author{A. Kontogeorgis}
\address{Department of Mathematics, National and Kapodistrian University of Athens\\
Panepistimioupolis, 15784 Athens, Greece
}
\email{kontogar@math.uoa.gr}

\author{G. Petroulakis}
\email{georgios.petroulakis.1@city.ac.uk}

\begin{abstract}
  The biduality and reflexivity theorems are known to hold for projective varieties defined over fields of characteristic zero, and to fail in positive characteristic. 
  In this article we construct  a notion of reflexivity and biduality in positive characteristic by generalizing the ordinary tangent space to the notion of $h$-tangent spaces. The ordinary reflexivity theory can be recovered as the special case $h=0$,  of our theory.
  Several varieties that are not ordinary reflexive or bidual  become  reflexive in our extended theory. 
\end{abstract}



\date{\today}
\maketitle



\section{Introduction}

The aim of this article is to study the biduality theorem and the stronger notion of reflexivity of varieties in positive characteristic. If $k$ is an algebraically closed field of characteristic $p\geq 0$, it is a very old observation that points in the projective space $\mathbb{P}_k^n$ correspond to hyperplanes in the dual projective space and vice versa. This notion of duality can be generalized to closed irreducible varieties $M\subset\mathbb{P}_k^n$ and gives rise to a dual variety $M^*$ in the dual projective space. 

The biduality theorem (known to hold over fields of characteristic 0) asserts that $\left(M^*\right)^* = M$. One of the proofs of this fact, \cite[p.29]{gelfand2008discriminants} uses the notion of the conormal bundle, the natural symplectic structure on the cotangent bundle of a manifold.  Wallace, \cite{wal}, was the first to consider the theory of projective duality over fields of positive characteristic. For a nice introduction to projective duality independent of the characteristic of the base field we refer to \cite{Kl:86}.

Let  $M \in \mathbb{P}^n_k$
\nomenclature{$M$}{a projective variety}
be a projective variety and denote by $M_{\mathrm{sm}}$ the set of smooth points of $M$.
The classical conormal variety $\mathrm{Con}(M)$ is defined by
\nomenclature{$C(M)$}{The conormal variety of $M$}
\[
\mathrm{Con}(M):=\overline{
\{ (P,H)
 \in M_{\mathrm{sm}} \times \mathbb{P}_k^{n*}: T_PM \subset H \}
} \subset M \times \mathbb{P}_k^{n*} \subset
\mathbb{P}_k^{n} \times \mathbb{P}_k^{n*},
\]
i.e., the Zariski closure of the algebraic set consisted of pairs $(P,H)$, $P\in M_{\mathrm{sm}}$,  $H\in \mathbb{P}^{*n}$ such that  $T_PM \subset H$.

Let $\pi_2$ be the second projection $\mathrm{Con}(M) \rightarrow \pi_2(\mathrm{Con}(M)):=M^*  \subset \mathbb{P}^{*n}$,
\nomenclature{$M^*$}{The projective dual of $M$}
which will be called the conormal map.
It is known that $M^*$ is an algebraic variety of $\mathbb{P}^{*n}$. If $\mathrm{Con}(M)= \mathrm{Con}(M^*)$, then $M$ is called reflexive. Equivalently, in terms of isomorphisms, $M$ is reflexive if the natural isomorphism from $\mathbb{P}_k^{n}$ to $\left(\mathbb{P}_k^{n^*}\right)^*$ induces the isomorphism 
\[
\xymatrix@R-0.5pc{
	\mathbb{P}_k^{n^*} \times \left(\mathbb{P}_k^{n^*}\right)^* 
	\ar@{=}[r]
	&
   \mathbb{P}_k^{n^*} \times \mathbb{P}_k^{n}
\\
\mathrm{Con}(M^*) \ar@{^{(}->}[u] \ar@{<->}[r]^-{\;\;\cong}
&
\mathrm{Con}(M) \ar@{^{(}->}[u] 
}
\]  
It is known that reflexivity implies biduality, but there are examples known of bidual varieties that are not reflexive.
Reflexivity also holds for all projective varieties in characteristic zero, while in characteristic $p> 0$, reflexivity can fail, see the Fermat-curve example in \cite{wal}. In positive characteristic there is the following criterion for reflexivity, whose proof may be found in \cite{hefklei}.

\begin{theorem}[Monge-Segre-Wallace]
A projective variety $M$ is reflexive  if and only if the conormal map $\pi_2$ is separable.
\end{theorem}

The problems of biduality and reflexivity of a projective variety $M\subset \mathbb{P}^n$ have been addressed by several authors via the use of the Gauss map, i.e., the rational map from $M$ to the Grassmann variety
$\mathbb{G}(n,m)$, which sends a smooth point $P\in M$ to the $m$-dimensional tangent space $T_PM\in\mathbb{P}^n$ - in the case of a hypersurcace, the Gauss map is just a map $\gamma: M \rightarrow \mathbb{P}^{* n}$. As proved in \cite{kaji}, the separability of the Gauss map and the reflexivity of a variety are equivalent in the one-dimensional case, i.e. for projective curves. For higher dimensions, the authors in \cite{kleipie} showed that 
the Gauss map of a projective variety $M$ is separable if $M$ is reflexive. On the other hand, the converse of this result, i.e. whether the reflexivity of a projective variety implies the separability of the Gauss map, was answered recently negatively, since there are  specific examples (such as the Segre varieties) for which this assumption is not true. These examples and further analysis is found in \cite{fukkaji} and the references therein. 
The previous work and results are, to the best of our knowledge, the most recent with regard to the study of biduality and reflexivity and are focused on weather and when they fail or not, in positive characteristic.
\bigskip

The aim of this article is to extend the notions of biduality and reflexivity in the case of positive characteristic. We will make appropriate definitions which will make some important examples of varieties reflexive. We generalize the theory of Lagrange varieties presented in \cite[p.29]{gelfand2008discriminants} for projective varieties in the zero characteristic case, by introducing the respective $h$-cotagent bundle and $h$-Lagrangian subvarieties. 
The case of hypersurfaces is illuminating and  straightforward calculations can be made in terms of the   implicit-inverse function theorem approach of Wallace, \cite{wal}.

Reflexivity has many important applications to enumerative geometry, computations with discriminants and resultants, invariant theory, combinatorics etc. We hope that our construction will find some similar applications to positive characteristic algebraic geometry. 
\bigskip

From now on $k$ is an algebraically closed field of positive characteristic $p$ and $q=p^h$ is a power of $p$. 
 Instead of tangent hyperplanes, we will consider generalized hyperplanes, i.e. hypersurfaces of the form $V(\sum_{i=0}^n a_i x_i^{p^h})$ and  the duality will be expressed  in terms of these generalized hyperplanes.

Let $V$ be a finite dimensional vector space over $k$.
Consider $M\subset \mathbb{P}(V)$ an irreducible  projective variety and consider the cone $M' \subset V$ seen as an affine variety in $V$. 
Assume that the homogeneous ideal of $M'$ is generated by the homogeneous polynomials $F_1,\ldots, F_r$. 
Fix a natural number $h$  and 
consider the $n+1$-upple 
\[
\nabla^{(h)}F_i = \left(
\left.  D_0^{(h)} \right|_P F_i,
\left.  D_1^{(h)}  \right|_P F_i, 
\ldots, 
\left.  D_n^{(h)} \right|_P F_i
\right),
\]
where $D_i^{(h)}$ denotes the $h$-Hasse derivative which will be defined in  definition \ref{HasseDerDef}.
Each $F_i$ defines a $p^h$-linear form given by 
\[
L_i^{(h)}:=\sum_{\nu=0}^n \left( \left. D_i^{(h)} \right|_P F_i  \right) x_\nu^{p^h}.
\]
For the precise definition of $p^h$-linear forms and their space $V^{*h}$ see section \ref{duala}.

For a projective variety $M$ we will define the set $M^h_{\mathrm{sm}}$ of smooth $h$-points in definition \ref{h non singular}, which set if non-empty is dense in $M$, since we have assumed that $M$ is irreducible.  
\begin{definition} 
For a projective irreducible variety $M$ with $M^h_{\mathrm{sm}}\neq \emptyset$ we define
the $h$-tangent space $T^{(h)}_PM$ at $P$
to be  the variety defined  by the equations $L_i^{(h)}=0$.
The $h$-conormal space is defined as the subset of $V\times V^{*h}$ 
\begin{equation} \label{conormal}
\mathrm{Con}^{(h)}(M)=
\overline{
\left\{ (P,H): P\in M^h_{\mathrm{sm}}, H \text{ is a } p^h-\text{linear form which vanishes on } T_P^{(h)}M
\right\}
}.
\end{equation}
\end{definition}
It is evident that the $h$-conormal space can be identified to the space of
$p^h$-linear forms on the $h$-normal space $N^{(h)}_M(P)$ defined as 
\[
N^{(h)}_M(P) =V/T_P^{(h)}(M).
\]
If the variety is not reflexive, we might choose an appropriate $h$ so that we can have a form of reflexivity based on $\mathrm{Con}^{(h)}(M)$. How are we going to select $h$? 
If the characteristic of the base  field $k$ is zero or if the variety $M$ is reflexive  then $h=0$.     If the variety $M$ is just a hypersurface then the answer is simple: If $M$ fails to be reflexive then the second projection $C(M)\rightarrow M^*$ is a map of inseparable degree $p^h$, and in this way we obtained the required $h$. 

Even in the case of hypersurfaces one has to be careful. Projective duality depends on Euler's theorem on homogeneous polynomials, since a homogeneous polynomial can be reconstructed by the values of all first order derivatives. An appropriate generalization of Euler's theorem is known but we have to restrict ourselves to a class of polynomials which we will call $h$-homogeneous. Their precise definition is given in definition \ref{h-homogeneous}.

The main result of our work is the following 
theorem 


\begin{theorem}\label{lagrdual}
Consider the vector space $V^{*h}$ of $p^h$-linear forms. 
Let $M\in \mathbb{P}(V)$ be an irreducible, reduced projective variety 
generated by $h$-homogeneous elements,
 which also has a non-empty $h$-non-singular locus, as these are defined in Definition \ref{h-homogeneous} and Definition \ref{h non singular}, respectively. 
Assume that we can select an $h$ so that the 
 map $\pi_2:V\times V^{*h} \supset \mathrm{Con}^{h}(M) \rightarrow \pi_2(M):=Z \subset V^{*h}$ is separable and generically smooth and also that 
 $Z$ has a non-empty set of $h$-non-singular points. Consider the map
\begin{eqnarray*}
\Psi:
V\times V^{*h} & \rightarrow &    V^{*h} \times \left(V^{*h}\right)^{*h} \\
(x,y) & \mapsto & (y,F(x)),
\end{eqnarray*}
where $F$ is the isomorphism $F:V \rightarrow \left( V^{*h}\right)^{*h}$ introduced in Theorem \ref{h-duality}.
Then 
\[
\Psi(\mathrm{Con}^{(h)}(M))=\mathrm{Con}^{h}(Z)\subset V^{*h} \times (V^{*h})^{*h}=V^{*h}\times V.
\]
\end{theorem}
Notice also  that in contrast to ordinary situation where the set of non-singular points forms a dense open subset, for $h>0$ the set of $h$-non-singular points can be empty. The existence of a non-empty set of $h$-non-singular points is essential for the definition of the conormal space and has to be assumed.   




Notice that the explicit construction of the dual variety involves a projection map which can be computed using  elimination theory, see  \cite[ex. 14.8 p. 315]{Eisenbud:95}. The algebraic set $M \subset \mathbb{P}^n$, gives rise to the conormal scheme $\mathrm{Con^{(h)}(M)} \subset \mathbb{P}^n\times \left(\mathbb{P}^{n}\right)^{h*}$. If $k[\xi_0,\ldots,\xi_n]$ is the polynomial ring corresponding to the dual projective space and 
\[
I=\langle f_1,\ldots, f_r\rangle \lhd k[x_0,\ldots,x_n],
\] the ideal corresponding to $M$, then the ideal $I'\lhd k[x_0,\ldots,x_n, \xi_0,\ldots,\xi_n]$ corresponding to the conormal scheme is generated by $I\cdot k[x_0,\ldots,x_n,\xi_0,\ldots,\xi_n]$ and the equations  
\[
\sum_{i=0}^r  \sum_{j=0}^n \lambda_i  D_{x_i}^{(h)} f_i \cdot \xi_i^{p^h}=0, \;\; \lambda_i \in k.
\] 
The dual variety can be computed by  eliminating the variables $x_0,\ldots,x_n, \lambda_1,\ldots,\lambda_r$ and by obtaining a homogeneous ideal in $k[\xi_0,\ldots,\xi_n]$. Notice that there are powerful algorithms for performing elimination using the theory of Gr\"obner bases, see example \ref{FermatChar0}.
\bigskip

The structure of the article is as follows:
In Section \ref{sec:HasseDer} we define and describe a number of important tools, notions and results, we are going to use through out the paper. First we start with the family of Hasse derivatives, which will be seen as derivatives with respect to some new ghost variables $x_i^{(q^h)}$. 
These derivatives were first introduced by Hasse and Schmidt \cite{Hasse1937},\cite{Schmidt39} in order to study Weierstrass points in positive characteristic. Afterwards,  we define the so-called $p^h$-linear forms and their respective space.  In the same section we define the $q$-symplectic form we are going to use in the last section, in order to create a suitable Lagrangian variety for our work. In the same section we generalize the Euler identity for homogeneous polynomials and obtain the $h$-homogeneous polynomial definition. 
In Section \ref{sec:implicit} we present the implicit-inverse function theorem approach of our theory, we make connections with elimination theory, and treat the hypersurface case. In the last section, we generalize all the respective notions met in Lagrangian manifold theory for biduality in characteristic zero, \cite[p.29]{gelfand2008discriminants} and using them we prove Theorem \ref{lagrdual}.

\section{Tools and basic Constructions} The main idea behind our approach, assuming that $k$ has characteristic $p>0$, is to set the quantity $x_i^{p^h}$  as a new variable $x_i^{(h)}$, for $h=0,1,2,...$. As it is well known, the classical partial derivatives $D_{x_i}$ on the polynomial ring $k[x_0,\ldots,x_r]$ are zero on the polynomials of the form $f(x_0^p,\ldots,x_r^p)$, and this is the reason biduality and reflexivity fail in positive characteristic. The theory of Hasse derivatives will help us deal with this.
\subsection{Hasse Derivatives}
\label{sec:HasseDer}
\begin{definition} \label{HasseDerDef}
A Hasse family of differential operators on a commutative unital $k$-algebra $A$, is a family $D_{\underline{n}}, \underline{n}\in \mathbb{N}^{r+1}$, of $k$-vector space endomorphisms of $A$ satisfying the conditions:
\begin{enumerate}
  \item $D_{\underline{0}}=\mathrm{Id}$
  \item $D_{\underline{n}}(c)=0$, for all  $c\in k$  and $  \underline{n}\neq \underline{0}$.
  \item $D_{\underline{n}} \circ D_{\underline{m}}=\binom{\underline{n}+\underline{m}}{\underline{n}}D_{\underline{n}+\underline{m}}$
  \item $D_{\underline{n}}(a\cdot b)=\sum_{\underline{i}+\underline{j}=\underline{n}}D_{\underline{i}}a \cdot D_{\underline{j}}b$,
\end{enumerate}
where for $\underline{n}=(n_0,\ldots,n_r),\underline{m}=(m_0,\ldots,m_r)\in \mathbb{N}^{r+1}$
\[
\binom{\underline{n}}{\underline{m}}=\binom{n_0}{m_0}\cdots\binom{n_r}{m_r}.
\]
\end{definition}
An example of a Hasse family is given as follows:
For $A=k[\underline{x}]=k[x_0,\ldots,x_r]$, and $\underline{x}^{\underline{m}}=x_0^{m_0}\cdots x_r^{m_r}$ we define
\[
D_{\underline{n}} \underline{x}^{\underline{m}}=
\binom{\underline{m}}{\underline{n}} \underline{x}^{\underline{m}-\underline{n}}.
\]
Let us denote by $D_i=D_{\underline{n}_i}$ for $\underline{n}_i=(0,\ldots,0,1,0\ldots,0)$, i.e. there is an $1$ in the $i$-th position. For general $\underline{n}$ we  can recover $D_{\underline{n}}$ by
$
D_{\underline{n}}=D_0^{n_0}\circ \cdots \circ D_r^{n_r},
$
where $D_i^{n_i}$ denotes the composition of $D_i$ $n_i$ times. One can prove (see
\cite{hefez89}) that for $n=\sum_{j=0}^s n_j p^j$ with $0\leq n_j<p$ for all $j=0,\ldots,s$ we have
\begin{equation} \label{for-dual-e}
D_i^n=\frac{1}{n_0!\cdots n_s!} \left(D_i^{p^s}\right)^{n_s} \cdots (D_i^p)^{n_1} (D_i^1)^{n_0},
\end{equation}
therefore for each $i$, the family $(D_i^n)$, $n\in \mathbb{N}$ is determined by the operators $D_i^1, D_i^p,D_i^{p^2},\ldots$

\begin{definition} \label{def5}
We will denote by $D_{x_i}^{(h)}$ the operator $D_i^{p^h}$.
\end{definition}

If $D_i^ja=0$ for some $a$ and $j\in \mathbb{N}$, then $D_i^m=0$ for all $m \geq_p j$. In particular if $D_i^{p^{\mu}}=0$, then $D_i^{p^\mu+1}(a)=\cdots =D_i^{p^{\mu+1}-1}(a)=0$.

The following result, \cite{hefez89}, will be used several times during derivation processes in the next sections.
\begin{lemma} \label{frob-rep}
Let $x$, $t$ we indeterminates and $q=p^h$. If $f(t)\in k[t]$, then
\[
D_x^n f(x^q)=
\begin{cases}
D_t^{n/q}(f)(x^q) & \mbox{ if } q \mid n \\
0 & \mbox{ if } q \nmid n
\end{cases},
\]
where $D_x^n$ (resp. $D_t^n$) are the Hasse derivatives defined on $k[x]$ (resp. $k[t]$). 
\end{lemma}

\begin{remark}
Note that in multilinear algebra, a system of divided powers on a $k$-algebra $A$,
is a collection of functions $x \mapsto x^{(d)}$ satisfying a set of axioms given in \cite[p. 579]{Eisenbud:95}. We observe that the Hasse derivatives $D_i^n$ form a system of divided powers on the commutative ring of differential operators $k[\partial/\partial x_i]$.
\end{remark}

\subsection{Semilinear algebra}
\label{duala}

Since first order Hasse derivatives can not grasp the structure of $p$-powers, we have to generalize the notion of tangent space. 


\subsubsection{Frobenius actions and Hilbert's 90 theorem} \label{sec:FrobAct}

We will now collect some results on Galois descent for vector spaces, see also \cite[lemma 2.3.8]{GilleCentralSimpleAlgebras}.
Let $V$ be a vector spare of dimension $n+1$, with a basis $B=\{e_0,\ldots,e_{n}\}$ and let $x_0,\ldots,x_n \in V^*$ be linear independent coordinates with respect to the basis $B$, i.e.
\[
V \ni v =\sum_{i=0}^n x_i(v) e_i.
\]
Let $F_{p}:k \rightarrow k$ be the Frobenius map $x\mapsto x^p$. The Frobenius map 
$F_p$ generates $\mathrm{Gal}(\bar{\mathbb{F}}_p/\mathbb{F}_p)\cong \widehat{\mathbb{Z}}$ as a profinite group. 
Assume that the space $V$ is equipped with  an semilinerar action of the Galois group 
$\mathrm{Gal}(\bar{\mathbb{F}}_p/\mathbb{F}_p)$, i.e. $\sigma(\lambda\cdot v)=\lambda^p \sigma(v)$ for all $\lambda \in k$ and $v\in V$. This action  is expressed by 
an $n \times n$ matrix $\rho(\sigma)\in \mathrm{GL}_n(k)$ for every 
$\sigma\in \mathrm{Gal}(\bar{\mathbb{F}}_p/\mathbb{F}_p)$.
The entries $a_{ij}$ of the matrix $\rho(\sigma)$ are given by 
\[
\sigma e_i= \sum_{\nu=0}^n a_{\nu,i} e_\nu.
\] 
In this way we see that $\rho:\mathrm{Gal}(\bar{\mathbb{F}}_p/\mathbb{F}_p) \rightarrow \mathrm{GL}_n(k)$ is not a homomorphism but a cocycle, i.e. it satisfies the condition:
\[\rho(\sigma \tau)=\rho(\sigma) \rho(\tau)^\sigma.\]
Hilbert's 90 theorem \cite{NeukirchCohoNumFields} assures that there is an $(n+1)\times (n+1)$ matrix $P$
such that $\rho(\sigma)=P^{-1}P^\sigma$ and thus we can find a basis which is trivial under the action of $\mathrm{Gal}(\bar{\mathbb{F}}_p/\mathbb{F}_p)$. The action of the Frobenius map with respect to this basis 
is the action given by the Frobenius map on the coordinates, i.e.
\begin{equation} \label{Frob-coord}
\xymatrix@R-1pc{
 V \ar[r]^{F_p} & V \\
v=\sum_{\nu=0}^n \lambda_i e_i \ar[r] \ar[d]  & F_p(v)=\sum_{\nu=0}^n 
\lambda_i^p e_i  \ar[d]\\
(\lambda_1,\ldots,\lambda_n) \ar[r] & (\lambda_1^p,\ldots,\lambda_n^p) \\
\bar{\mathbb{F}}^n \ar[r] & \bar{\mathbb{F}}^n.
}
\end{equation}
From now on we will use Frobenius invariant bases $e_0,\ldots,e_n$.
The polynomial ring $k[x_0,\ldots,x_n]$ is  naturally attached to the vector space $V$ since $\mathrm{Sym}(V^*)=k[x_0,\ldots,x_n]$.

\begin{remark}
For an element $v\in V$ we will denote by $v^{p^i}$ the element $F_p^{i}(v)$ for $i\in \mathbb{Z}$. Since we use Frobenius invariant bases we can work with coordinates as we did in eq. (\ref{Frob-coord}). 
\end{remark}

\begin{definition} \label{h-hyperplane}
An $h$-hyperplane $H$ is the algebraic set given by an equation of the form:
\[
\sum_{i=0}^n a_i x_i^{p^h}=0, a_i\in k.
\]
\end{definition}

Such a hyperplane defines a $p^h$-linear map:
\begin{eqnarray*}
\phi: V & \rightarrow & k \\
\sum_{i=0}^{n} x_i(v) e_i=v & \mapsto &  \phi(v)=\sum_{i=0}^n a_i x_i(v)^{p^h}.
\end{eqnarray*}
The set of $p^h$ linear maps denoted by $V^{*h}$ consists of functions $\phi: V \rightarrow k$, such that 
\begin{enumerate}
\item $\phi(v_1+v_2)=\phi(v_1)+ \phi(v_2)$ for all $v_1,v_2 \in V$
\item $\phi(\lambda v)=\lambda^{p^h} \phi(\lambda)$ for all $\lambda\in k$ and $v\in V$.
\end{enumerate}
The space $V^{*h}$ becomes naturally a $k$-vector space, with basis the set $\{x_i^{p^h}: 0\leq i \leq n\}$.
\begin{theorem}
\label{h-duality}
The space $\left(V^{*h}\right)^{*h}$ is canonically isomorphic to the initial space $V$.  
\end{theorem} 
\begin{proof}
The element $v\in V$ is sent by the isomorphism $F$ to the space $\left(V^{*h}\right)^{*h}$ defined by:
\begin{eqnarray*}
F:V & \rightarrow & \left(V^{*h}\right)^{*h} \\
v & \mapsto & F(v),
\end{eqnarray*}
where $F(v)$ is the map defined by:
\begin{eqnarray*}
F(v): V^{*h} & \rightarrow & k \\
 \phi & \rightarrow & F(v)(\phi)=\phi\left(v^{1/p^{2h}}\right)^{p^h}.
\end{eqnarray*}
Notice that the map $v \mapsto v^{1/p^h}$ is well defined since the field $k$ is assumed to be perfect. 

Observe first that $F(v)$ is indeed a $p^h$-linear map. Indeed, 
\begin{eqnarray*}
F(v)( \lambda_1\phi_1+ \lambda_2\phi_2) &= &
\left( 
\lambda_1\phi_1\left(v^{1/p^{2h}}\right) +
\lambda_2\phi_2\left(v^{1/p^{2h}}\right)
\right)^{p^h} \\
& = & 
\lambda_1^{p^h} \phi_1\left(v^{1/p^{2h}}\right)^{p^h} +
\lambda_2^{p^h} \phi_2\left(v^{1/p^{2h}}\right)^{p^h} 
\\
& =& 
\lambda_1^{p^h} F(v)(\phi_1)+ \lambda_2^{p^h} F(v)(\phi_2).
\end{eqnarray*}
Now we prove that $F(v)$ is linear:
\begin{eqnarray*}
F(\lambda_1 v_1 + \lambda_2 v_2)(\phi)
& =& 
\phi
\left( \lambda_1^{1/p^{2h}} v_1^{1/p^{2h}} + \lambda_2^{1/p^{2h}} v_2^{1/p^{2h}}
\right)^{p^h} \\
& = & 
\lambda_1 \phi(v_1^{1/p^{2h}})^{p^h}+ \lambda_2 \phi(v_2^{1/p^{2h}})^{p^h} \\
& = & 
\left(
\lambda_1 F(v_1) +\lambda_2 F(v_2) 
\right)\phi,
\end{eqnarray*}
i.e., $
F(\lambda_1 v_1 +\lambda_2 v_2) = \lambda_1 F(v)+ \lambda_2 F(v),$
for all $\lambda_1,\lambda_2\in k$ and $v_1,v_2 \in V$.
\end{proof}
Let us work with coordinates now. Express an element 
$v\in V$ as $v=\sum_{i=0}^n x_i(v) e_i$, where $\{e_i\}_{i=0,\ldots,n}$ is a Frobenius invariant basis as expressed in section \ref{sec:FrobAct}, and let $\phi\in V^{*h}$ written in terms of a Frobenius dual basis as $\phi=\sum_{i=0}^n y_i(\phi) e_i^{*h}$. Set $q=p^h$, we  have 
\begin{equation}
\label{eq:contact}
\phi(v) =\sum_{i=0}^n x_i(v)^q y_i(v), 
\end{equation}
while we have (recall that $F(v)\in (V^{*h})^{h*}$)
\begin{equation}
\label{eq:dualcontact}
F(v)(\phi)=
\phi
\left(
\sum_{i=0}^n x_i(v)^{1/2q} e_i
\right)^q
=
\sum_{i=0}^n x_i(v) y_i(v)^q.
\end{equation}
This means that 
\begin{equation}
\label{PSI-map}
\xymatrix{
  V \times V^{*h} \ni (v,\phi)  \ar@{|->}[r]
  \ar_-{\Psi}[d] & \phi(v)=\sum_{i=0}^n x_i^q y_i \in k
  \\
   V^{*h} \times (V^{^*h})^{*h} \ni (\phi, F(v))
   \ar@{|->}[r] 
   & F(v)(\phi)=\sum_{i=0}^n x_iy_i^q.
}
\end{equation}
In our generalised point of view 
duality means that a point $[v] \in \mathbb{P}(V)$, represented by the vector $v\in V$, can be also seen as a $q$-hyperplane $[F(v)]$ on
$\mathbb{P}\left( (V^{*h})^{*h} \right)$.  


\subsection{q-Symplectic forms}
\label{q-symplecti-forms}

Let $\mathbb{F}$ be a field of positive characteristic $p$ and let $q=p^h$ be a certain power of $p$. In order to define a suitable Lagrangian variety in the positive characteristic case, we need its respective symplectic form.   
\begin{definition}
A $q$-symplectic form $\Omega$ on $V$ is a function:
\[
\Omega: V \times V^{*h} \rightarrow k
\]
which is additive, i.e. for all $v_1,v_2,w_1,w_2 \in V$ we have
\[
\Omega(v_1+v_2,w_1)=\Omega(v_1,w_1)+ \Omega(v_2,w_1), \qquad 
\Omega(v_1,w_1+w_2)=\Omega(v_1,w_1)+ \Omega(v_1,w_2),
\]
such that there is 
a symplectic basis $\{e_1,\ldots,e_n,f_1,\ldots,f_n\}$
so that  
\[
\Omega(e_i,e_j)=0=\Omega(f_i,f_j), \Omega(e_i,f_j)=\delta_{ij}, \Omega(f_i,e_j)=-\delta_{ij}.
\]
Moreover for  arbitrary elements 
\[
v=\sum_{i=0}^n \lambda_i e_i + \sum_{j=0}^n \mu_j f_j 
\]
and 
\[
w=\sum_{i=0}^n \lambda_i' e_i + \sum_{j=0}^n \mu_j' f_j
\]
the symplectic form is computed:
\[
\Omega(v,w)=\sum_{i=0}^n \left( \lambda_i^{p^h} \mu_i' - \mu_i \lambda_i'^{p^h} \right).
\]
\end{definition}
\begin{remark}
As in \cite[p. 8]{Da_Silva2001-jk}, the notions of  $p^h$-orthogonality, $p^h$-symplectic, $p^h$-isotropic and $p^h$-langrangian subvector spaces can be defined. Since these notions are not needed in this note, we will not develop their theory here.  
\end{remark}

 \subsection{Powers of Frobenius as ghost variables}
\label{sec:ghost}

In this section we will add extra ghost variables $x_i^{(h)}$ for $0\leq i \leq n$ and for $h=1,\ldots,\infty$. This is an idea comming from the similarities of the $p$-power Frobenius map and differential equations \cite[sec. I.1.9]{GossBook} and the ring of differential polynomials see \cite[exam. 5.2.5]{CrespoDiff}. 

\begin{lemma} \label{onto-phi}
Consider a term $\underline{x}^{\underline{i}}$, where $\underline{i}=(i_0,\ldots i_n) \in \mathbb{N}^{n+1}$, and the $p$-adic expansions of each index:
\[
i_\nu=\sum_{\mu=0}^\infty i_\nu(\mu) p^\mu, \qquad 0 \leq i_\nu(\mu)<p.
\]
Therefore, a term $\underline{x}^{\underline{i}}$ can be written as
\begin{equation} \label{padic-exp}
\underline{x}^{\underline{i}}=
\prod_{\mu_0=0}^\infty \cdots 
\prod_{\mu_n=0}^\infty
x_0^{i_0(\mu_0) p^\mu}\cdots x_n^{i_n(\mu_n) p^\mu}.
\end{equation}

Consider the ring
\begin{equation} \label{R-def}
R:=k[x_0,\ldots,x_n,x_0^{(1)},\ldots,x_n^{(1)},\ldots,x_0^{(h)},\ldots,x_n^{(h)},\ldots,]
\end{equation} 
and define the degree $\deg x_\nu^{(i)}=p^i$. 
We also define the homomorphism
\begin{eqnarray}
   \phi:R &\rightarrow& k[x_0,\ldots,x_n] \label{phi-def}\\
x_i^{(j)} & \mapsto & x_i^{p
^j} \mbox{ for all } 0\leq i \leq n, 0\leq j \leq h. \nonumber
\end{eqnarray}
The map $\phi$ is onto, and moreover
\begin{equation} \label{Hasse-diff}
\phi
\left(
D_{x_i^{(j)}}f
\right)=
D_{x_i}^{(h)} \phi(f).
\end{equation}
\nomenclature{$D_{x_i}^{(h)} \phi(f)$}{The Hasse derivative of $f$ with respect to $x_i^{p^h}$ }
\end{lemma}
\begin{proof}
Let $f\in k[x_0,\ldots,x_n]$. If we write every term of $f$ as in eq. (\ref{padic-exp}) and replace $x_j^{ i_j(\mu) p^{\mu}}$
by $\left( x_j^{(\mu)} \right)^{i_j}$,  we get a polynomial $\tilde{f}\in R$ such that $\phi(\tilde{f})=f$.
The relation given in eq. (\ref{Hasse-diff}) follows by the property of the Hasse derivative
\[
D_{x_i}^{(h)}(x_j^{p^\ell})=\delta_{ij} \delta_{h,\ell},
\]
and the differentiation rules.
\end{proof}

In other words, this lemma shows that if we set the quantity $x_i^{p^h}$ which appears in the related varieties, as a new variable $x_i^{(h)}$, with the use of suitable expansions, the partial derivation $D_{x_i^{(h)}}$ with respect to the variables $x_i^{(h)}$ will coincide with the Hasse derivatives $D_{x_i}^{(h)}$.

\begin{remark}
The kernel of the map $\phi$ of eq. (\ref{phi-def}) is the ideal generated by $x_i^{p^h}-x_i^{(h)}$, which is a homogeneous ideal by the definition of the degrees $\deg_\nu x^{(i)}$. Therefore, we have the following compatible diagram of vector spaces, rings and derivations:
\[
\xymatrix{
\tilde{V} \ar[d]  & R=\mathrm{Sym}(\tilde{V}^*) \ar^{\phi}[d] &  \{ D_{x_i^{(h)}} \;\; 0\leq i \leq n, 0\leq h \leq N\}   \ar@{=}[d]\\
V & k[x_0,\ldots, x_n]=\mathrm{Sym}(V^*) & \{ D_{x_i}^{(h)} \;\; 0\leq i \leq n, 0\leq h \leq N
\}
}
\]
In the above diagram we have a vector space, the natural ring of polynomial functions on it and the natural set of derivations. When taking the quotient by the ideal $\ker \phi$, the set of derivations is not altered and the derivations corresponding to the dual basis of $\tilde{V}$ survive, giving rise to Hasse derivations on the quotient.

\end{remark}

\begin{remark} \label{remark16} The definition of the ring $R$ in this subsection, could provide an alternative way to force separability and therefore reflexivity to hold, for a class of weighted projective varieties, which we may call bihomogeneous.

Consider an ideal $I$  of $k[x_0,\ldots,x_n]$ generated by elements $F_1,\ldots,F_t$. 
Instead of working with the polynomial ring $R$, of infinite Krull dimension  we restrict ourselves 
to the ring 
\[
R_N:=k[x_0,\ldots,x_n,x_0^{(1)},\ldots,x_n^{(1)},\ldots,x_0^{(N)},\ldots,x_n^{(N)},]
\]
where $N$ is big enough so that the map $\phi_h:R_h \rightarrow k[x_0,\ldots,x_n]$ is onto $I$. Essentially this means that every term of all polynomials $F_i$ is of the form $x_0^{i_0}\cdots x_n^{i_n}$ and the $p$-adic expansions of $i_j$, $0\leq j \leq n$ do not involve $p$-powers $p^{h}$ with $N<h$. For example for $p=3$ and the polynomial 
$x^{10}+x^{21}$ we have to take $N=2$ since 
\[
x^{10}+x^{21}=x^{1+3^2}+x^{3\cdot7}=x x^{3^2}+(x^3)^7=
x \phi(x^{(2)}) + \phi(x^{(1)})^7. 
\]   
Let now 
 $\tilde{I}$ be 
the ideal of $R_N$
defined by $\phi^{-1}(I)$, then
$\tilde{I}$ is generated by the polynomials $\tilde{F}_i\in R_N$
defined in the proof of Lemma \ref{onto-phi}.
Since, the procedure of Lemma \ref{onto-phi} replaces all powers of the form $x_i^{p^h}$ by the new coordinates $x_i^{(h)}$, which still have degree $p^h$, if $I$ is a homogeneous ideal of $k[x_0,\ldots,x_n]$, then it is generated by homogeneous elements $F_1,\ldots,F_t$ and the corresponding polynomials in new variables are still homogeneous. In other  words, if $I$ is a homogeneous ideal of  $k[x_0,\ldots,x_n]$,         then $\tilde{I}$
is a homogeneous ideal of $R$.

Recall that a weighted projective space is the quotient $\mathbb{P}(a_0,\ldots,a_n)=
\left(\mathbb{A}^{n+1}-\{0\}\right)/k^*$ under the equivalence relation $
(x_0,\ldots,x_n) \sim (\lambda^{a_0} x_0, \ldots, \lambda^{a_n} x_n),$
for $\lambda \in k^*$.
 
In our case, in order to form algebraic sets corresponding to ideals $\phi^{-1}(I)$, we have to consider the weighted projective spaces,
 $\mathbb{P}(1,\ldots,1,p,\ldots,p,p^2,\ldots,p^2,\ldots,p^N,\ldots,p^N)$.
In a weighted projective space linear equations of the form 
\[
\sum_{h=0}^N \sum_{i=0}^n a_{h,i} x_i^{(h)}=0,
\]
do not give rise to homogeneous ideals unless they are of the form 
\[
\sum_{i=0}^n a_{h_0,i} x_i^{(h)}=0, 
\]
and it is not entirely clear what projective duality will mean for weighted projective varieties. Of course, it is known that every weighted projective variety $M$ is isomorphic to an  ordinary projective variety $\tilde{M} \in \mathbb{P}^\ell$ for some big enough element $\ell$, \cite[th. 4.3.9]{Hosgood2016-sp}. The homogeneous ideal $\tilde{I}$ corresponding to $M$ is generated by polynomials of degree smaller than $p$, therefore it is reflexive. 

We will not pursue here the theory of duality of weighted projective varieties, but we can see something interesting for some of them; if we
consider the polynomial ring 
\[
R_{0,N}=k[x_0,\ldots,x_n,x_0^{(1)},\ldots,x_n^{(1)},\ldots,x_0^{(N)},\ldots,x_n^{(N)}],
\]
but now $\deg (x_0)^{(i)}=1$ for all $1\leq i \leq h$, the ideal 
$\phi^{-1}(I) \in R_{0,N}$, of a homogeneous ideal $I$ of $k[x_0,\ldots,x_n]$ does not need to be homogeneous in $R_{0,N}$ with this grading.
If it is homogeneous, then we can define it as bihomogeneous. For example, the hypersurface defined by the polynomial $\sum_{i=0}^n x_i^{p^h+1}$ gives rise to the ideal generated by the polynomial $
 x_0^{(h)} x_0+x_1^{(h)} x_1+\cdots+x_n^{(h)} x_n, 
$
which is bihomogeneous. On the other hand, the hypersurface defined by the homogeneous polynomial 
$
x_0^{p+1}-x_1x_2\cdots x_{p+1}
$, is not bihomogeneous, i.e., the polynomial 
$
x_0^{(1)}x_0-x_1x_2 \cdots x_{p+1}
$
is homogeneous in the graded ring $R_N$ but not in the graded ring $R_{0,N}$.
Observe now that the projective algebraic set $V(\phi^{-1}(I)) \subset \mathbb{P}^{h(n+1)}$ defined by the bihomogeneous ideal $\phi^{-1}(I) \subset R_{0,N}$ does not have a variable raised to a power of $p$, therefore it is reflexive. 
\end{remark}

\subsubsection{Example: Generalized quadratic forms}
Let $\underline{x}=(x_0,\ldots,x_n)^t$ and consider the homogeneous polynomial
\[
f_A:=\underline{x}^t A \underline{x}^q=\sum_{i,j=0}^n x_i a_{ij} x_j^q,
\]
where $A=(a_{i,j})$ is an $(n+1)\times (n+1)$ matrix, and $q=p^h$.
If $A=\mathbb{I}_{n+1}$, then $F$ is the diagonal Fermat hypersurface also called Hermitian hypersurface.
For $q=1$ the polynomial $f_A$ is just a quadratic form.

We compute that for a point $P=[a_0:\ldots:a_n] \in V(f_A)$
\[
D^{(0)}_{x_\ell} f_A(P)= \sum_{j=0}^n a_{\ell,j} a_j^{p^h}
\]
and
\[
D^{(h)}_{x_\ell} f_A(P) =\sum_{i=0}^n  a_i a_{i,\ell}.
\]

We write the coordinates of $P$ as a column vector $a=(a_0,\ldots,a_n)^t$ and
we compute both $\nabla f_A$, $\nabla^q f_A$,
\[
\nabla f_A=(D_{x_0}f_A,\ldots,D_{x_n} f_A)=
Aa^q=(A^{1/q}a)^q
\]
and
\[
\nabla^q f_A=(D^{(q)}_{x_0}f_A,\ldots,D^{(q)}_{x_n} f_A)=a^t\cdot A.
\]
The Gauss map $a \mapsto (A^{1/q} a)^q$ is inseparable.

Define $\xi=(\xi_0,\ldots,\xi_{n})^t$ and
$\xi^{(q)}=
\left(\xi_0^{(q)},\ldots,\xi_n^{(q)}
\right)^t$, given by
\[
\xi=\nabla f_A=A a^q \mbox{ and } \xi^{(q)}=
\left(
\nabla^q f_A
\right)^t=
\left(a^t\cdot A\right)^t=A^t a.
\]

We will now introduce ghost variables in order to force reflexivity. 
Here we consider the variables $x^q=(x_0^q,\ldots,x_n^q)^t=y=(y_0,\ldots,y_n)^t$ as a set of new variables $y$ and we write the homogeneous polynomial defining the variety as
\[
F_A=x^t A y =\sum_{i,j=0}^n x_i a_{ij} y_j.
\]
The Gauss map in this case is given by:
\[
(a,b)\mapsto \nabla F_A= (A\cdot b, A^t a).
\]
If for a point $(a,b)^t \in V(F_A)$ satisfying $a^t F_A b=0$ we introduce the variables $\xi=A\cdot b,\xi_1=A^t\cdot a$,
then the point $(\xi,\xi_1)$ satisfies the equation:
\[
\xi_1^t A^{-1} \xi=0
\]
since
\[
\xi_1^t A^{-1} \xi=a^t A A^{-1} A b=a^t A b=0.
\]
Observe that the value $\xi_1^q=A^{qt}\cdot a^q$ can be explicitly expressed in terms of the variables
$\xi$ by the equation:
\[
A A^{-tq}\xi_1^q=A A^{-tq} \cdot (A^{tq})  a^q =Aa^q=\xi.
\]
Notice also that the map $\phi:(X,Y)\mapsto (A^t  Y,A X)=(\xi_1,\xi)$ and similarly the map $\psi:(\xi_1,\xi)\mapsto (A^{-1} \xi,A^{-t} \xi_1)$ and  $\psi\circ \phi=\phi\circ \psi=\mathrm{Id}$.

Let $M=V(f_A) \subset \mathbb{P}(V)$ and $\tilde{M}=V(F_A) \subset \mathbb{P}(\tilde{V})$.
The conormal variety $C(\tilde{M}) \subset \mathbb{P}(\tilde{V}) \times \mathbb{P}(\tilde{V}^*)$ is given by the pairs $(a,b; \xi,\xi_1)=(a,b; A\cdot b, A^t\cdot a)$. In order to compute the  conormal variety $C(M) \subset \mathbb{P}(V) \times \mathbb{P}(V^*)$ we pass from $\tilde{M}$ to $M$ by imposing the relation $b=a^q$ and we obtain
$(a,a^q; A\cdot a^q,A^t \cdot a)$. Observe that $\xi,\xi_1$ satisfy the equation of the dual
\[
\xi^{qt}A^{-1} \xi=0.
\]

\subsection{Variants of Euler theorem}

The Euler identity for homogeneous polynomials implies that for a homogeneous polynomial $F(x_0,\ldots,x_n)\in k[x_0,\ldots,x_n]$, of degree $\deg F$ we have
\[
\sum_{i=0}^n x_i D_{x_i}^{p^0} F(x_0,\ldots,x_n)=\deg F \cdot F(x_0,\ldots,x_n).
\]
If $p\mid \deg F$, a lot of information is lost. In particular
the first order partial derivations $D_{x_i}^{p^0}F$ can be  zero. Next proposition allows us to get some information, from the higher  derivatives $D_{x_i}^{p^i}$. We need the following 




\begin{proposition}
\label{q-Euler}
Let $q=p^h$ be a power of the characteristic.
Let $P_j(x_0,\ldots,x_n)$,$Q_j(x_0,\ldots,x_n)$ 
be polynomials in $k[x_0,\ldots,x_n]$, $j=1,\ldots,s$, where $P_j$ are homogeneous of degree $n=:\deg_h(f)$, and $Q_j$ have no indeterminate raised to a power bigger than or equal to a power of $q$.
 If
\begin{equation} \label{h-Euler}
f(x_0,\ldots,x_n) =\sum_{j=1}^s P_j(x_0^q,\ldots,x_n^q)Q_j(x_0,\ldots,x_n),
\end{equation}
then
\[
\sum_{i=0}^n x_i^q D_{x_i}^qf(x_0,\ldots,x_n)=\deg_h(f) \cdot f(x_0,\ldots,x_n).
\]
\end{proposition}
\begin{proof}
\cite[prop. 3.10]{hefez89}
\end{proof}
\begin{definition} \label{h-homogeneous}
We will call a polynomial $h$-homogeneous of degree $\deg_h(f)$ if it is a linear combination of polynomials given in eq. (\ref{h-Euler}) of the same degree. 
\end{definition}

\begin{remark}
A polynomial which is homogeneous and $h$-homogeneous is bihomogeneous according to  remark \ref{remark16}.
\end{remark}


\subsection{$h$-Tangent and $h$-cotangent spaces and bundles}
\label{sec:TangentBundle}

In order to compare our definition of $h$-tangent space we recall here the classical definition.

Let $M$ be a projective variety defined by 
$h$-homogeneous polynomials  $F_1,\ldots,F_t$, as these were defined in Definition \ref{h-homogeneous}, generating the homogeneous ideal $I$. Let $S$ be the algebra $k[x_0,\ldots,x_n]/I$.

\begin{definition}
Let $P=[a_0:\ldots:a_n]$ be a point on $M$. The tangent space $T_P M$
of $M$ at $P$, is defined as the zero space of the differentials (we will denote by $D_{x_\nu}^{(0)}$ the classical derivative according to definition \ref{def5}). In other words,
\begin{equation} \label{system-tangent}
dF_i=\sum_{\nu=0}^n D_{x_\nu}^{(0)}F_i(P) x_\nu \mbox{ for all } 1\leq i \leq t,
\end{equation}
\[
T_PM=V( \langle dF_1,\ldots,dF_t \rangle).
\]
\nomenclature{$T_P M$}{The tangent space $T_P M$
of $X$ at $P$}
\end{definition}

\begin{definition} \label{dual}
For every  $f\in R$, define the differential form on the tangent space $T_P M$:
\begin{equation} \label{df-form}
df:=\sum_{\nu=1}^n D_{x_\nu}^{(0)}f(P) x_\nu,
\end{equation}
\nomenclature{$df$}{Differential of an element in the coordinate ring}
which gives rise to elements in the dual space $T_P^* M$,
by sending a solution $(x_0:\cdots:x_r)\in T_PM$ of system (\ref{system-tangent}) to the value $df$ given in eq. (\ref{df-form}).
\end{definition}
\nomenclature{$T_P^* M$}{Cotangent space of $M$ at $P$}
 The element $df$ is well defined, since  if $f_1-f_2 \in \langle F_1,\ldots,F_t \rangle$, then the differentials $df_1,df_2$ introduce the same linear form on $T_P M$, see  \cite[chap II. sec. 1]{ShafaBo1}.
Let $\mathcal{O}_X(P)$ be the ring of functions defined at $P$. The map 
\[
d: k[x_0,\ldots,x_n] \rightarrow (T_PM)^*
\]
defines an isomorphism of $\mathfrak{m}_P/\mathfrak{m}_P^2$ to $(T_PM)^*$, \cite[chap II. th. 2.1]{ShafaBo1}. This fact implies that the dimension of the tangent space is invariant under isomorphism, see \cite[chap II. Cor. 2.1]{ShafaBo1}.

\subsubsection{$h$-Tangent bundles}
For every $F\in R$
we define the $h$-linear form:
\begin{eqnarray*}
L_F^{(h)}: V & \longrightarrow k &  \\
\sum_{\nu=0}^n x_i(v) e_i=  v  & \mapsto & \sum_{\nu=0}^n D_{x_\nu}^{(h)}F(P) x_\nu^{p^h}(v).
\end{eqnarray*}
\begin{definition} \label{tan-cotan-def}
Let $M$ be defined in terms if the homogeneous ideal
$\langle F_1,\ldots,F_t\rangle$.
For $h\geq 0$, the  $h$-tangent space $T_P^{(h)}M$ at $P \in M$
\nomenclature{$\Theta_PM$}{Complete tangent space of $M$ at $P$}
is defined by 
\[
T_P^{(h)} M = \bigcap_{i=1}^t \mathrm{ker} L_{F_i}^{(h)} \subseteq V.
\]
It is clear from the definition that $T_P^{(h)}M$ is a $k$-vector space.
\end{definition}
\begin{remark}
The notion of classical tangent space is independent of the isomorphism class of a variety. If $\Phi:M \rightarrow Y$ is a local isomorphism from a  Zariski neighbourhood $U$ of $P$ to a Zariski neighbourhood $V$ of $\Phi(P)$,  then
 $\dim_k T^{(0)}_PM=\dim_k T^{(0)}_{\Phi(P)}Y$.

This does not hold for the case of the $h$-tangent spaces, the space $T_P^{(h)}M$ depends on the embedding of $M$ in an ambient space. For example the affine space $\mathbb{A}^1=\mathrm{Spec}(k[x])$ has one dimensional tangent space $T_P^{(h)} \mathbb{A}^1$ for all $h>1$,  while its isomorphic image $\mathrm{Spec}(k[x,y]/\langle x\rangle)\subset \mathbb{A}^2$ has 2-dimensional $h$-tangent space for all $h>1$.  

Of course, in order to correct this, one can strengthen the notion of isomorphism $\Phi:X\rightarrow Y$, by requiring that $\Phi$ induces an isomorphism to $h$-tangent spaces as well.  
\end{remark}



\begin{remark}
As R. Vakil observes \cite[chap. 12]{VakilSea}, the quantity $\sum_{i=0}^n D_{x_i}^{(0)} F \cdot x_i$ is the linear part of a given polynomial $F\in k[x_0,\ldots,x_n]$. In a similar fashion $\sum_{i=0}^n D_{x_i}^{(h)} F \cdot x_i^{p^h}$ is the $p^h$-linear part of the polynomial $F$, that is all terms that can be written as $\left(\sum_{i=0}^n a_i x_i\right)^{p^h}$, $a_i\in k$.
\end{remark}

\begin{definition}
The $h$-cotangent space $T_P^{(*h)}M$ for $h\geq 0$
at $P$ is defined as the vector space  generated by the elements (set $q=p^h$) 
\begin{equation}
\label{diff-q}
d^{(h)}f = \sum_{\nu=0}^{n} D_{x_\nu}^{(h)}f(P) x_\nu^q
\end{equation}
for elements $f \in k[x_0,\ldots,x_n]/I(M)$. Notice that the expression $d^{(h)}f$ defined for $f$ as above gives rise to a well defined form on the tangent space. Moreover $d^{(h)} x_i^q$ is an element in $T_P^{(*h)}M$.
\end{definition}
\begin{remark}
Similar to  ordinary differentials, the map
given in eq. (\ref{diff-q})
is well defined, i.e., if $f_1-f_2\in I(M)$, then $d^{(h)}f_1-d^{(h)}f_2$ is the zero map on the tangent space $T^{(h)}M$. In this way the differentials in eq. (\ref{diff-q}) define functions
\[
\phi:T_P^{(h)} M \rightarrow k.
\]
\end{remark}
Let us consider the example $M=\mathrm{Spec}k[x,y]/\langle x \rangle$. In order to compute the  $h$-cotangent space $T_0^{(*h)}M$ let us compute 
$
d^{(h)}f(x,y)
$
for $f(x,y)\in k[x,y]$. The possible outcomes are all expressions of the form  $ax^q+b y^q$, $a,b\in k$. Notice that for $f \in \langle x \rangle$ we have $d^{(h)}x=0$, so the $h$-cotangent space is two dimensional.



Let $R$ be a finite presented $k$-algebra. 
A Frobenius map on $R$ is a ring homomorphism $\Phi:R \rightarrow \Phi(R)\subset R$, such that $\Phi(\lambda x)=\lambda^{p} \Phi(x)$, for all $\lambda\in k$ and $x\in R$.
The image $\Phi(R)$ is a subring of $R$. In this way we form a sequence of nested subrings of $R$, 
\[
R \supset \Phi(R) \supset \Phi^{2}(R)  \supset \Phi^{3}(R) \supset \cdots
\] 
If $R$ is a local ring with maximal ideal $\mathfrak{m}$ then all rings $\Phi^hR$ are 
local rings as well, with maximal ideals $\mathfrak{m}^{(h)}=\Phi^h(\mathfrak{m})$.
If $f:R_1\rightarrow R_2$ is a ring homomorphism of two rings equipped with Frobenius maps $\Phi_1,\Phi_2$ respectively then we require
\[
f( \Phi^h_1 R_1) \subset \Phi^h_2 (R_2).
\]
If moreover $f$ is a local homomorphism of local rings $R_i$ with corresponding maximal ideals $\mathfrak{m}_i$, for $i=1,2$ then $f(\mathfrak{m}_1^{(h)}) \subset \mathfrak{m}_2^{(h)}$.
In what follows we will consider the polynomial ring $k[x_0,\ldots,x_r]/I$, and the localizations at certain maximal ideals of the ring $k[x_0,\ldots,x_n]/I$.

\begin{definition}
The intrinsic $h$-cotangent space $\Theta_P^{(*h)}M$ is defined to be the space $\mathfrak{m}_P^{(h)}/\mathfrak{m}_P^{(h)^2}$ and equals to to the cotangent space of the local ring $\Phi^h(\mathcal{O}_P)$. 
\end{definition}

\begin{lemma} \label{cotmax}
The space $\Theta_P^{(*h)}M\cong \mathfrak{m}_P^{(h)}/\left(\mathfrak{m}_P^{(h)}\right)^2$ is a subspace of 
$T_P^{(*h)}M$.  
\end{lemma}
\begin{proof}
Consider the space of $h$-linear forms in Hasse derivatives given by differentials
$d^{(h)}f: T^{(h)}\rightarrow k$. 
Consider  the map 
\[
d^{(h)}:\frac{\mathfrak{m}_P^{(h)}}{
\left(\mathfrak{m}_P^{(h)} \right)^2}
\longrightarrow T_P^{(*h)}(M),
\]
we will prove it is injective. 

Let $G\in  k[x_0^{p^h},\ldots,x_n^{p^h}]$
 be a polynomial representative of an  element in $\mathfrak{m}_p^{(h)}$ so that  $d^{(h)}G$ is the zero form on $T_P^{(h)}M$.  Then $d^{(h)}G$ is a linear combination $\sum_{\nu=0}^t \lambda_\nu d^{(h)}F_\nu$ of the forms
$d^{(h)}F_\nu$ for $F_\nu$, $1\leq \nu \leq n$, generating the homogeneous ideal of $M$. Now the difference   $G-\sum_{\nu=0}^t \lambda_\nu d^{(h)}F_\nu \in
 \left(\mathfrak{m}_P^{(h)} \right)^2$, since it has not terms of the form $x_\nu^q$,  and the result follows. 
\end{proof}

\begin{remark}
By eq. (\ref{diff-q}) we have that $d^{(h)}(x_i^{p^h})$ on the tangent space acts like the $q$-form 
$x_i \mapsto x_i^q$. 
\end{remark}


\begin{corollary}
The dimension of $\Theta^{(*h)}_PM$ is an invariant of the isomorphism class of a variety, i.e. if $\Phi:M \rightarrow Y$ is a local isomorphism from a  Zariski neighbourhood $U$ of $P$ to a Zariski neighbourhood $V$ of $\Phi(P)$,  then $\dim_k \Theta^{(h)}_PM=\dim_k \Theta^{(h)}_{\Phi(P)}Y$.   
\end{corollary}

Let $M\subset V$ be an irreducible variety. Consider the algebraic set  $\Theta \subset V\times M$ consisting of pairs $(a,P)\in V\times M$ such that $a$ is $h$-tangent at $P$. The second projection  $\pi:\Theta \rightarrow M$  is onto and has fibres the spaces $\Theta^{(h)}_PM$. By \cite[Chap. I.63 th.7]{ShafaBo1}  we have that $\dim_k \Theta_P^{(h)}M \geq s$ for all $P\in M$ and equality is attained at a non empty open subset of $M$. 

\begin{definition} \label{h non singular}
We will say that a point $P \in M$ is $h$-non-singular if 
\[
\dim_k \Theta_P^{(h)}= \dim_k T_P^{(*h)}M=\dim X.
\]
\end{definition}
\subsubsection{Differential between tangent spaces.}
Consider the projective varieties $V\subset \mathbb{P}^n$, $W\subset \mathbb{P}^m$ defined in terms of the  homogeneous ideals $\langle f_1,\ldots,f_r \rangle \in k[x_0,\ldots,x_n]$ and $\langle g_1,\ldots,g_s
\rangle \in k[y_0,\ldots,y_m]$ respectively. A map $F:V \rightarrow W$ is given by polynomials $F_0,\ldots,F_m \in k[x_0,\ldots,x_n]$ such that $y_i=F_i(x_0,\ldots,x_n)$ for $i=0,\ldots,m$.
Set 
\[
J_{0,h}(f_1,\ldots,f_r)=
\begin{pmatrix}[ccc|ccc]
D_{x_0}^{(0)} f_1 & 
\cdots 
&
D_{x_n}^{(0)} f_1 &
D_{x_0}^{(h)} f_1 &
\cdots
&
D_{x_n}^{(h)} f_1
\\
\vdots & & \vdots & \vdots & & \vdots
\\
D_{x_0}^{(0)} f_r & 
\cdots 
&
D_{x_n}^{(0)} f_r &
D_{x_0}^{(h)} f_r &
\cdots
&
D_{x_n}^{(h)} f_r
\end{pmatrix}
=(A|A')
\]
and similarly
\[
J_{0,h}(g_1,\ldots,g_s)=
\begin{pmatrix}[ccc|ccc]
D_{y_0}^{(0)} g_1 & 
\cdots 
&
D_{y_m}^{(0)} g_1 &
D_{y_0}^{(h)} g_1 &
\cdots
&
D_{y_m}^{(h)} g_1
\\
\vdots & & \vdots & \vdots & & \vdots
\\
D_{y_0}^{(0)} g_s & 
\cdots 
&
D_{y_m}^{(0)} g_s &
D_{y_0}^{(h)} g_s &
\cdots
&
D_{y_m}^{(h)} g_s
\end{pmatrix}
=(B|B').
\]
The kernel of the matrix $A$  at $P$ (resp. $B$ at $F(P)$) corresponds to the ordinary tangent space of $V$ (resp. $W$) while the kernel of $A'$ (resp $B'$) corresponds to the $h$-tangent space.

By substitution of $y_i=F(x_0,\ldots,x_n)$ for $0\leq i \leq m$ in $g_1,\ldots,g_s$ we write each $g_1,\ldots,g_s$ as an element in the ideal $\langle f_1,\ldots,f_r \rangle$. Therefore elements in $T_P V$, resp. $T_P^{(h)}V$, given as elements in the kernel of $A$ (resp. $A'$) are sent to elements in $T_PW$, resp. $T_P^{(h)}W$.

Consider now the matrix (observe that $y_i^{p^h}=F_i(x_0,\ldots,x_n)^{p^h}$)
\begin{align*}
J_{0,h}(F_1,\ldots,F_n,
F_1^q,\ldots,F_n^q
) &=
\begin{pmatrix}[ccc|ccc]
D_{x_0}^{(0)} F_0 & 
\cdots 
&
D_{x_n}^{(0)} F_0 &
D_{x_0}^{(h)} F_0 &
\cdots
&
D_{x_n}^{(h)} F_0
\\
\vdots & & \vdots & \vdots & & \vdots
\\
D_{x_0}^{(0)} F_m & 
\cdots 
&
D_{x_n}^{(0)} F_m &
D_{x_0}^{(h)} F_m &
\cdots
&
D_{x_n}^{(h)} F_m
\\
\hline
D_{x_0}^{(0)} F_0^{p^h} & 
\cdots 
&
D_{x_n}^{(0)} F_0^{p^h} &
D_{x_0}^{(h)} F_0^{p^h} &
\cdots
&
D_{x_n}^{(h)} F_0^{p^h}
\\
\vdots & & \vdots & \vdots & & \vdots
\\
D_{x_0}^{(0)} F_m^{p^h} & 
\cdots 
&
D_{x_n}^{(0)} F_m^{p^h} &
D_{x_0}^{(h)} F_m^{p^h} &
\cdots
&
D_{x_n}^{(h)} F_m^{p^h}
\end{pmatrix}
\\
&=
\begin{pmatrix}[ccc|ccc]
D_{x_0}^{(0)} F_0 & 
\cdots 
&
D_{x_n}^{(0)} F_0 &
D_{x_0}^{(h)} F_0
&
\cdots
&
D_{x_n}^{(h)} F_0
\\
\vdots & & \vdots & \vdots & & \vdots
\\
D_{x_0}^{(0)} F_m & 
\cdots 
&
D_{x_n}^{(0)} F_m &
D_{x_0}^{(h)} F_m &
\cdots
&
D_{x_n}^{(h)} F_m
\\
\hline
0 & 
\cdots 
&
0 &
D_{x_0}^{(h)} F_0^{p^h} &
\cdots
&
D_{x_n}^{(h)} F_0^{p^h}
\\
\vdots & & \vdots & \vdots & & \vdots
\\
0 & 
\cdots 
&
0 &
D_{x_0}^{(h)} F_m^{p^h} &
\cdots
&
D_{x_n}^{(h)} F_m^{p^h}
\end{pmatrix}
\\
&=
\begin{pmatrix}[c|c]
J & J' \\
\hline
0 & J^{p^h}
\end{pmatrix}.
\end{align*}
The chain rule implies
\begin{align*}
\begin{pmatrix}[ccc|ccc]
D_{x_0}^{(0)} g_1 & 
\cdots 
&
D_{x_m}^{(0)} g_1 &
D_{x_0}^{(h)} g_1 &
\cdots
&
D_{x_m}^{(h)} g_1
\\
\vdots & & \vdots & \vdots & & \vdots
\\
D_{x_0}^{(0)} g_s & 
\cdots 
&
D_{x_m}^{(0)} g_s &
D_{x_0}^{(h)} g_s &
\cdots
&
D_{x_m}^{(h)} g_s
\end{pmatrix}
&=
(B|B')\begin{pmatrix}[c|c]
J & J' \\
\hline
0 & J^{p^h}
\end{pmatrix}
\\
&=(BJ | BJ'+B'J^{p^h}).
\end{align*}
An element $\bar{a}=(a_0,\ldots,a_n)^t \in T_{F(P)}W $, $\bar{b}=(b_0,\ldots,b_n)^t \in T_{F(P)}^{(h)}W$ by definition of the tangent spaces satisfies
\[
BJ\bar{a}=0 \qquad 
(BJ'+B' J^{p^h})\bar{b}=0.
\]
On the other hand we have
\[
\begin{pmatrix}[c|c]
J & J' \\
\hline
0 & J^{p^h}
\end{pmatrix}
\begin{pmatrix}
\bar{a} \\
\bar{b} 
\end{pmatrix}=
\begin{pmatrix}
J\bar{a}+J'\bar{b} \\
J^{p^h} \bar{b}
\end{pmatrix}
\]
therefore $B(J\bar{a}+J'\bar{b})=0$ and $B' J^{p^h} \bar{b}=0$. This allows us to write the differentials: 
\[
dF: T_P(V) \rightarrow T_{F(P)}W \text{ and }
dF^{(h)}: 
T^{(h)}_P(V) \rightarrow T^{(h)}_{F(P)}W
\]
as follows
\[
dF(\bar{a})=J\bar{a}+ J' \bar{b} \text{ and }
dF^{(h)}(\bar{b})=J^{p^h} \bar{b}.
\]
When $\bar{b}=0$ is the zero $h$-tangent vector then $dF$ is the classical map. The differential in the $h$-tangent space is independent on the choice of $\bar{a} \in T_P^{(0)}V$.


 We have proved the following:
\begin{proposition}
\label{p-power-diff}
Let $F:V \rightarrow W$ be a map between polynomial varieties, expressed in terms of polynomials $F_0,\ldots,F_s$. 
Then the $p^h$-power of the ordinary differential $T^{(0)}_PV\rightarrow T^{(0)}_{F(P)}W$ is the natural map $T^{(h)}_PV \rightarrow T^{(h)}_{F(P)}V$. 
\end{proposition}

\subsection{Vector fields and  differential forms}

We will now define vector fields as differential operators in terms of  Hasse-derivatives.  The identification 
\[
\frac{\mathfrak{m}_P}{\mathfrak{m}_P^2} \stackrel{d}{\longrightarrow} T_P^* M 
\]
proves that $dx_0,\ldots,dx_n$ give a basis of the cotangent space, since 
$\mathfrak{m}_P/\mathfrak{m}_P^2$ is generated as vector space by the classes of $x_0,\ldots,x_n$
modulo $\mathfrak{m}_P^2$. Also in the classical case the partial derivatives 
$\partial/\partial x_i$ give rise to naturally dual elements, i.e. elements in $T_P M$.

Let us assume that the variety $M$ has a non-empty open set of $h$-non-singular points. 
On this open set we will employ the identification
$
\frac{
\mathfrak{m}_P^{(h)}
}
{
\left(
\mathfrak{m}_P^{(h)}
\right)^2
} 
\stackrel{d^{(h)}}{\longrightarrow} 
T_P^{(*h)} M$ of lemma \ref{cotmax},
which sends 
\[
\frac{
\mathfrak{m}_P^{(h)}
}
{
\left(
\mathfrak{m}_P^{(h)}
\right)^2
} 
\ni m=\sum_{i=0}^n a_ix_i^{p^h}
\mapsto
d^{(h)}m = \sum_{i=0}^{p^h} a_i d^{(h)}x_i^q \in T_P^{(*h)} M.
\]

\begin{definition}
A vector field $X$ is a sum
\begin{equation} \label{vector-field-X}
X=
\sum_{h=0}^{\infty}
 \sum_{i=0}^n
a_{h,i}(X) D_{x_i}^{(h)},
\end{equation}
where all but finite coefficients $a_{h,i}(X)$ are zero.
\nomenclature{$X$}{Vector field as a sum of Hasse differential operators}
The elements $a_{h,i}(X)$ are coefficients in $\mathcal{O}_M$, depending linearly on $X$.
Vector fields form  $\mathcal{O}_M$-modules.
\nomenclature{$\mathcal{O}_M$}{Sheaf of regular functions on $M$}
\end{definition}

\begin{definition}
For every  $i\in \{0,\ldots,n\}$ we define the differential form $d^{(h)} x_i^q$, seen as a formal symbol.
This definition can be given a functorial interpretation, by considering the
module of $p$-graded K\"ahler differentials as a universal object representing the functor of Hasse derivations, see \cite[chap. 16]{Eisenbud:95}. 
\end{definition}

For a function $f\in \mathcal{O}_M(U)$ we define the differentials 
$d^{(h)}f$ (with respect to Hasse derivatives, see also eq. (\ref{diff-q})):
\begin{equation} \label{differen-f}
d^{(h)}(f)=
\sum_{i=0}^n D_{x_i}^{(h)}(f)
d^{(h)} x_i^{p^h}.
\end{equation}
Recall the notation $q=p^h$ and note that from eq. (\ref{differen-f}) we see that $d(x_i^{q})=d^{(h)}x_i^q$ which can be seen as an element in $T_P^{(*h)}M$. Without the Hasse derivatives, the
differential $d(x^{q})$, when computed in terms of eq. (\ref{differen-f})  is  zero, but here it is a generator of the alternating algebra of differential forms.

\begin{definition}
For $q_i=p^i$ define the formal monomials $d^{(h_1)} x_{i_1}^{q_1} \wedge d^{(h_2)} x_{i_2}^{q_2} \wedge \cdots \wedge d^{(h_j)} x_{i_j}^{q_{i_j}}$ of degree $j$, where for monomials $m,n$ of degrees $k$ and $l$ we have
\[
m \wedge n = (-1)^{kl} n \wedge m.
\]
A differential form of degree $i$ is a formal linear combination of monomials of degree $p$, with coefficients from $\mathcal{O}_X(U)$. 
 
We also require that for a function $f$ we have
\begin{equation} \label{qraise}
f dx_i \wedge d^{(h)} x_j^{p^h}= dx_i \wedge f^{p^h} d^{(h)}x_j^{p^h}. 
\end{equation}
The above requirement is natural since that alternating algebras are defined as quotients of the tensor algebra, see \cite[Apendix 2]{Eisenbud:95} of an ordinary form by a $p^h$-form. We will use this definition in lemma \ref{proveLagrangian} in order to prove that the $h$-conormal space is Lagrangian. 

A derivation of degree $s\in \mathbb{Z}$ on $\mathcal{O}_M(U)$ is a $k$-linear operator sending a form of degree $j$ to a form of degree $j+s$ such that
\[
D(\omega \wedge \tau)= D\omega \wedge \tau + (-1)^{sj} \omega \wedge D\tau.
\]
\end{definition}
We will need the following  derivations.
\begin{enumerate}
\item
 The derivations $d^{(h)}$ of degree $+1$, such that $d^{(h)}f$ is given by eq. (\ref{differen-f})
and $d^{(h')}d^{(h)}=0$ for all $h,h'\in \mathbb{N}$.
\item
 The derivation $i_X$
\nomenclature{$i_X$}{Replacement derivation corresponding to vector field $X$}
of degree $-1$ corresponding to
vector field $X$, given by $i_X(\mathcal{O}_X)=0$ while for $X$ given by eq. (\ref{vector-field-X}) and $\omega$ given by
\[
\omega=\sum_{h=0}^{\infty} \sum_{i=0}^r b_{h,i}(\omega) d^{(h)} x_i^q,
\mbox{ for }
b_{h,i}(\omega)\in \mathcal{O}_X(U) \mbox{ we have}
\]
\begin{equation} \label{twisted-act-on-diff-forms}
i_X(\omega)=\sum_{h=0}^\infty \sum_{i=0}^r 
\left(
a_{h,i}(X)^{p^h} b_{h,i}^{p^h}(\omega)
\right).
\end{equation}
\end{enumerate}
\begin{remark}
A vector field is a section of the tangent bundle, i.e. for every $P \in X$  if the functions $a_{h,i}$ are in
$\mathcal{O}_X(U)$ for an open set $U$ containing $P$, then
the evaluation of $a_{h,i}$ at $P$ gives us a tangent vector in  $T_P M$,
\begin{equation} \label{eval-vf}
X(P)=\sum_{h=0}^{\infty} \sum_{i=0}^r
a_{h,i}(X)(P) D_{x_i}^{(h)}.
\end{equation}
Indeed, using the $i_X$ derivation we see that the vector field $D_{x_i}^{(h)}$ is the dual basis element to the differential form $d^{(h)}x_i^q$. Thus, the evaluated vector field gives rise to an element in the dual space of  $T_P^* M$.
\end{remark}

Assume now that the maximal ideal at $P \in M$ is generated by $t_1,\ldots,t_s$, and consider the differentials $dt_1,\ldots,dt_s$.

The classical  cotangent vector bundle (see \cite[p. 60]{ShafaBo2}) is the vector bundle
\[
T^*M=\bigoplus_{i=1}^r \mathcal{O}_M dt_i.
\]
A classical  differential form $\xi$ is given by
\[
\xi=\sum_{i=0}^r \xi_i dt_i, \;\; \xi_i \in \mathcal{O}_M.
\]
Keep in mind that a vector bundle in algebraic geometry over an open set  $U\subset M$ is described  in terms of $\mathbb{A}^r_U=\mathrm{Spec} \mathcal{O}_M(U)[\xi_1,\ldots,\xi_r]$, see \cite[ex. 5.18, p. 128]{Hartshorne:77}.

In analogy to the classical case, an $h$-differential form is given by 
\[
\xi^{h}=\sum_{i=0}^r \xi_i dt_i^{(h)}, \;\; \xi_i \in \mathcal{O}_M.
\]


%
%
\section{The case of hypersurfaces}
\label{sec:implicit}

In this section we focus on the hypersurface case. When the variety is given as the zero set of a single polynomial  we can use a form of implicit-inverse function theorem which allows us to express the coordinates $x_i$ as functions of the dual coordinates. This method works if the  $h$-Hessian is generically invertible.  
In characteristic zero we consider the hypersurface $V(f) \subset \mathbb{P}^n$ given by a polynomial $f$, if we set $\Xi_i=D_{x_i} f \in k[\underline{x}]$, we can find the ideal in $k[\underline{\Xi}]$ by eliminating the variables $\underline{x}$. Let us illustrate this method in  characteristic zero by the following 

\begin{example} \label{FermatChar0}
Consider the Fermat curve given as the zero locus of 
\[
x_0^5+x_1^5+x_2^5=0.
\] 
 This in magma \cite{Magma1997} can be done as follows:
If $y_i=D_{x_i}f$, we fist define the ideal 
\[
I=\langle
x_0^5 + x_1^5 + x_2^5,
    -5x_0^4 + y_0,
    -5x_1^4 + y_1,
    -5x_2^4 + y_2
   \rangle \lhd k[x_0,\ldots,x_2,y_0,\ldots,y_2],
\]
and then we eliminate the variables $x_0,x_1,x_2$ using the {\tt EliminationIdeal} function:
\[
J=\left\langle
\begin{array}{l}
 y_0^{20} - 4y_0^{15}y_1^5 - 4y_0^{15}y_2^5 + 6y_0^{10}y_1^{10} - 124y_0^{10}y_1^5y_2^5 +
        6y_0^{10}y_2^{10} - 4y_0^5y_1^{15} - 124y_0^5y_1^{10}y_2^5
\\
       - 124y_0^5y_1^5y_2^{10}
        - 4y_0^5y_2^{15} + y_1^{20} - 4y_1^{15}y_2^5 + 6y_1^{10}y_2^{10} - 4y_1^5y_2^{15} +
        y_2^{20}
\end{array}
\right\rangle.
\]
We can now consider the same elimination process, arriving at the ideal $J$
generated by the elements
\begin{eqnarray*}
g_1 & = & y_0^{20} - 4 y_0^{15} y_1^{5 }- 4 y_0^{15} y_2^{5 }+ 6 y_0^{10} y_1^{10} - 124 y_0^{10} y_1^{5 }y_2^{5 }+
    6 y_0^{10} y_2^{10} - 4 y_0^{5 }y_1^{15} - 124 y_0^{5 }y_1^{10} y_2^{5 }
      \\ & & - 124 y_0^{5 }y_1^{5 }y_2^{10} -
    4 y_0^{5 }y_2^{15} + y_1^{20} - 4 y_1^{15} y_2^{5 }+ 6 y_1^{10} y_2^{10} - 4 y_1^{5 }y_2^{15} + y_2^{20}
\\
g_2 & = & x_0 - 20 y_0^{19} + 60 y_0^{14} y_1^{5 }+ 60 y_0^{14} y_2^{5 }- 60 y_0^{9 } y_1^{10} +
    1240 y_0^{9 }y_1^{5 }y_2^{5 }-
\\ & &
    60 y_0^{9 }y_2^{10} + 20 y_0^{4 }y_1^{15} + 620 y_0^{4 }y_1^{10} y_2^{5 }+
    620 y_0^{4 }y_1^{5 }y_2^{10} + 20 y_0^{4 }y_2^{15} \\
g_3 & = & x_1 + 20 y_0^{15} y_1^{4 }- 60 y_0^{10} y_1^{9 }+ 620 y_0^{10} y_1^{4 }y_2^{5 }+ 60 y_0^{5 }y_1^{14} +
    1240 y_0^{5 }y_1^{9 }y_2^{5 }+ 620 y_0^{5 }y_1^{4 }y_2^{10} - \\ & & 20 y_1^{19} + 60 y_1^{14} y_2^{5 }-
    60 y_1^{9 }y_2^{10} + 20 y_1^{4 }y_2^{15} \\
g_4 & = & x_2 + 20 y_0^{15} y_2^{4 }+ 620 y_0^{10} y_1^{5 }y_2^{4 }- 60 y_0^{10} y_2^{9 }+ 620 y_0^{5 }y_1^{10} y_2^{4 }+
    1240 y_0^{5 }y_1^{5 }y_2^{9 }+ 60 y_0^{5 }y_2^{14} + \\ & & 20 y_1^{15} y_2^{4 }- 60 y_1^{10} y_2^{9 }+
    60 y_1^{5 }y_2^{14} - 20 y_2^{19}.
\end{eqnarray*}
Observe that the generators $g_2,g_3,g_4$ express $x_0,x_1,x_2$ as a function of $\underline{y}$, which follows by differentiating the defining equation $g_1$
of the dual hypersurface with respect to $y_0,y_1,y_2$, i.e., $x_i=D_{y_i}g_1$ for $i=0,1,2$.
After elimination in the ideal  $J$ of the variables $\underline{y}$ we arrive at the original equation as expected. 
\end{example}
Similarly, the implicit-inverse function method will allow us to solve ``locally'' and express $\Xi_i$ as functions of $k[x_0,\ldots,x_n]$. The problem with this method is that Zariski topology does not have fine enough open sets for the implicit (or the equivalent inverse) function theorem to hold. Actually this was one of the reasons for inventing etal\'e topology \cite[p. 11]{MilneLEC}. The approach of Wallace is based on defining algebraic functions in order for the implicit function theorem to work. 
We will follow the ideas of Wallace \cite[sec. 4.1]{wal}. Let $X_1,\ldots,X_n$ be a set of indeterminates of the field $k$. A separable algebraic function $\phi$ over $k(X_1,\ldots,X_n)$ will be called a $k$-function of $X_1,\ldots, X_n$. 
If $x_1,\ldots, x_n$ is any set of elements of $k$ and $y$ is a specialization of $\phi$ over the specialization $(X_1,\ldots,X_n)\mapsto (x_1,\ldots,x_n)$, then $y$ will be called a value of $\phi$ at $(x_1,\ldots,x_n)$, and will be written $y = \phi(x_1,\ldots,x_n)$. The partial derivative $\partial \phi/\partial X_i$ 
for each $i$, is a rational function of $X_1,\ldots, X_n$ and $\phi$. If this rational function
is defined at $(x_1,\ldots,x_n,y)$ (i.e. has non zero denominator), then the $k$-function $\phi$
will be called differentiable at $(x_1,\ldots, x_n, y)$.

\begin{remark}
If we allow $k$-functions then the duality theorems have a simpler form. For example for $(a,p)=1$ the dual curve of the Fermat curve
$
x_0^a+x_1^a+x_2^a=0
$
is the dual curve
$
x_0^{b}+x_1^{b}+x_2^{b}=0
$
such that $\frac{1}{a}+\frac{1}{b}=1$, see \cite[Example 2.3, p. 20]{gelfand2008discriminants}. 
\end{remark}

\begin{theorem}[Implicit function theorem]
\label{theorem:implicit}
If $x_0,\ldots,x_{2n}$ satisfy the $k$-functions $\phi_i(x_1,\ldots,x_{2n})=0$ for $i=1,\ldots,n$, differentiable at $(x_1,\ldots,x_{2n},0)$ and the Jacobian $n \times n$-matrix $(\partial \phi_i/\partial x_j)$ is invertible, then there are $k$-functions $f_0,\ldots,f_n$ of $y_0,\ldots,y_n$ such that $x_i=f(x_{n+1},\ldots, x_{2n})$ for all $1\leq i \leq n$. 
\end{theorem}
\begin{proof}
Therorem 6 in \cite{wal}.
\end{proof}
The above theorem in practice allows us to work with hypersurfaces as follows: Let $V(f)$ be a projective hypersurface. We put coordinates $(x_0,\ldots,x_n)$ on the space $\mathbb{P}^n$ and $y_0,\ldots,y_n$ on $\mathbb{P}^{* n}$. We have the equations:
\begin{equation} \label{usual-der}
y_i=\partial f /\partial x_i=\phi_{i}(x_0,\ldots,x_n).
\end{equation}
If the Hessian matrix $(\partial \phi_j/ \partial x_i)=(\partial^2 /f \partial x_i \partial x_j)$ is not singular, then the implicit function theorem allows us to express $x_i$ as $k$-functions of $y_0,\ldots,y_n$. 

For example, in characteristic zero (or if $p\nmid a-1$, the hypersurface defined by $f=\sum_{i=0}^n x_i^a$ has $y_i=\partial f/ \partial x_i=a x_i^{a-1}$, therefore $x_i=(y_i/a)^{\frac{1}{a-1}}$. The last expression is
in accordacne to theorem \ref{theorem:implicit}, since the Hessian matrix equals $a(a-1)\cdot\mathrm{diag}(x_0^{a-2},\ldots,x_n^{a-2})$, which is generically invertible. We can arrive to the dual hypersurface by replacing $x_i$ in the defining equation of $f$, i.e.
\[
\sum_{i=0}^n x_i(y_0,\ldots,y_n)^{a}=a(a-1) \sum_{i=0}^n y_i^{\frac{a}{a-1}}. 
\]
Notice that $b=\frac{a}{a-1}$ satisfies the symmetric equation $1/a+1/b=1$. 

If $p\mid a-1$, then the equation $y_i=a x_i^{a-1}$ does not allow us to express $x_i$ in terms of $y_i$. Keep in mind that the rational function field is not perfect, and we are not allowed to take $p$-roots of polynomials. 

Let $V(f)$ be a hypersurface corresponding to the irreducible homogeneous and $h$-homogeneous polynomial $f$ of degree prime to the characteristic. 
By equation (\ref{h-Euler}) we have  that if the Gauss map is not separable then $y_i=\partial f/ \partial x_i=g_i^{p^{h}}(\underline{x})$.  Moreover  
by Euler's theorem we have
\[
f=\deg(f)\cdot \sum_{i=0}^{n} x_i g_i(\underline{x})^{p^{h}}.
\]
In our approach we propose to consider instead of eq. (\ref{usual-der}) the equations 
\[
y_i=D_{x_i}^{(h)}(f).
\]
Then under the assumption that the ``Hessian'' $ D_{x_j}^{(0)} D_{x_i}^{(h)} f$ is invertible we can express 
\[
x_i=g_i(y_0,\ldots,y_n), 
\]
where $g_i$ is a $k$-function. 

\begin{remark}
Even in characteristic zero, the Hessian might be singular. Consider for example a hyperplane $V(\sum a_i x_i)$. The first derivatives are constants and the Hessian is zero. This situation is related to the case of singular Gauss map. For a detailed study of this case in terms of classical differential geometry see \cite{MR2014407}. 
\end{remark}
In some cases we can prove that the Hessian is invertible
\begin{lemma} \label{InvertibleHess}
Let $f$ be a homogeneous polynomial so that 
so that one at least of its derivatives $D_{x_i}^{(h)}f$, $0\leq i \leq n$ is not zero, and all non-zero $D_{x_i}^{(h)}f$ derivatives have degree $d_i$
 prime to the characteristic $p$. 
Then the  $(n+1)\times (n+1)$ matrix $D_{x_j}^{(0)} D_{x_i}^{(h)}(f)$ is generically invertible.
\end{lemma}
\begin{proof}
Assume that the above mentioned map is not invertible, then one collumn,  say the first one, is a linear combination of the other collumns, that is
\[
\begin{pmatrix}
D_{x_0}^{(0)} D_{x_0}^{(h)}(f) \\
\vdots\\
D_{x_n}^{(0)} D_{x_0}^{(h)}(f)
\end{pmatrix}
=\sum_{\mu=1}^n
\lambda_\mu
\begin{pmatrix}
D_{x_0}^{(0)} D_{x_\mu}^{(h)}(f) \\
\vdots\\
D_{x_n}^{(0)} D_{x_\mu}^{(h)}(f)
\end{pmatrix}
\]
Summing along each collumn after multiplying by $x_\nu$ and  using Euler's theorem  we have (set $d_\mu=\deg D_{x_{\mu}}^{(h)}f$)
\begin{equation}
\label{Dqzero}
  d_0 D_{x_0}^{(h)}(f)=\sum_{\mu=1}^n \lambda_n d_\mu D_{x_\mu}^{(h)}(f).
\end{equation}
Let $\delta_0$ be the degree of the polynomial $f$ in the variable $x_0^q$, $q=p^h$.  The above  eq. (\ref{Dqzero}) is impossible by considering the degrees of both sides in the variable $x_0^q$ unless $\delta=0$, in this case $D_{x_0}^{(h)}f=0$. This forces  the right hand side  of eq. (\ref{Dqzero}) to be zero, which allows us to repeat the above argument for an other variable $x_i$, until we prove  inductively that all derivatives $D^{(h)}_{x_i}$ are zero, a contradiction. 
\end{proof}

\begin{lemma}
\label{lemma41}
Consider a function $f$ as given in equation (\ref{h-Euler}) in proposition \ref{q-Euler}. Then this function satisfies the invertible Hessian criterion of lemma \ref{InvertibleHess} and the dual variety given by equation 
\[
G(y_0,\ldots,y_n)=
f\big(
x_0(y_0,\ldots,y_n),\ldots,x_n(y_0,\ldots,y_n)
\big)=0.
\]
\end{lemma}
Let us now consider the Hasse derivatives $D_{y_\nu}^{(h)}$ for $q=p^h$ of $G(\bar{y})=\sum_{i=0}^n x_i^q y_i$, where $x_i$ are considered as functions of $y_i$
\[
z_\nu:=
D_{y_\nu}^{(h)}
\left(
\sum_{i=0}^n x_i^q y_i
\right)=
\sum_{i=0}^n 
\left(
D_{y_\nu}^{(0)} x_i
\right)^q y_i. 
\]
We compute 
\begin{equation}
\label{dual-contact-form}
\sum_{\nu=0}^n z_\nu y_\nu^q=
\sum_{\nu=0}^n
\sum_{i=0}^n 
\left(
D_{y_\nu}^{(0)} x_i
\right)^q y_i y_\nu^q.
\end{equation}
Since $x_i(y_0,\ldots,y_n)$ is homogenous in the variables $y_0,\ldots,y_n$ as well the classical 
Euler identity gives us that 
\[
\sum_{\nu=0}^n y_\nu D_{y_\nu}^{(0)} x_i=c x_i
\text{ for some } c\in k,
\]
so eq. (\ref{dual-contact-form}) gives us 
\[
\sum_{\nu=0}^n z_\nu y_\nu^q=
c \sum_{i=0}^n x_i^q y_i=0. 
\]
This means that the point $\bar{x}=(x_0,\ldots,x_n) \in V(f)\subset \mathbb{P}^{n}$ has the $q$-hyperplane
\[
(X_0:\cdots: X_n)\in \mathbb{P}^n_k 
\text{ such that }
\sum_{\nu=0}^n y_i X_i^q=0
\]
with coordinates $(y_0,\ldots, y_n)$ as $h$-tangent and the point $(y_0,\ldots,y_n)\in V(G)\subset \mathbb{P}^{h*}$ has the $q$-hyperplane
\[
(Y_0:\cdots: Y_n)\in \mathbb{P}^n_k 
\text{ such that }
\sum_{\nu=0}^n z_\nu Y_\nu^q=0
\]
 with coordinates $(z_0,\ldots,z_n)$ as $h$-tangent.
Reflexivity esentially means that  the map 
\begin{eqnarray*}
V\times V^{*h} & \rightarrow &    V^{*h} \times \left(V^{*h}\right)^{*h} \\
(x,y) & \mapsto & (y,F(x))
\end{eqnarray*}
induces an isomorphism from $\mathrm{Con}^{(h)}(X)$ to $\mathrm{Con}^{(h)}(Y)$, where $F$ is the isomorphism $F:V \rightarrow \left( V^{*h}\right)^{*h}$ introduced in theorem \ref{h-duality}. For proving this we will require the notion of Lagrangian variety for algebraic sets defined over the field of complex numbers.



%
%
%
\subsection{Example: A class of Fermat hypersurfaces}

Let $p\neq 2$ be a prime. Consider the curve
\[
\sum_{i=0}^n x_i^{2p+1}=0.
\]
We set also $y_i=D_{x_i}^{(h)}f=2x_i^{p+1}$. We can express $x_i$ in terms of $y_i$, that is 
\[
x_i=\left(
\frac{1}{2} y_i
\right)^{\frac{1}{p+1}}
\]
The dual variety is then described as the zero set of the $k$-function
\[
G(y_0,\ldots,y_n)=\sum_{i=0}^n y_i^{\frac{2p+1}{p+1}}=0.
\]
We now compute the derivatives $z_i=D_{y_i}^{(h)}G=cx_i^{p-p^2+1}$ for some $c\in k$. 
We now expand
\begin{align*}
0=
\left(
\sum_{i=0}^{n} x_i^{2p+1}
\right)^{1+p-p^2}
&=
\left(\displaystyle
\sum_{i=0}^{2p+1} x_i^{2p+1}
\right)
\left(\displaystyle
\sum_{j=0}^{2p+1} x_j^{p(2p+1)}
\right)
\left(\displaystyle
\sum_{k=0}^{2p+1} x_k^{-p^2(2p+1)}
\right)
\\
&=
\displaystyle
\sum_{i=0}^n x_i^{(2p+1)(1+p-p^2)}+
\sum_{i=0}^n x_i^{2p+1}
\sum_{\substack{j=0 \\ j\neq i}}^n
x_j^{p(2p+1)}
\sum_{\substack{k=0 \\ k\neq i}}^n
x_k^{-p^2(2p+1)}
\\
&=
\displaystyle
\sum_{i=0}^n x_i^{(2p+1)(1+p-p^2)}
=
c^{-1}\sum_{i=0}^n z_i^{2p+1}.
\end{align*}
This proves that $(z_0,\ldots,z_n)$ are in $V(f)$.

\section{Lagrangian varieties}
\subsection{$h$-cotangent bundle and $h$-Lagrangian subvarieties}
\label{sec:LangVar}

The space $V\times V^{*h}$ can be identified to the $h$-cotangent bundle 
$T^{(*h)}(V)$ of $V$.
Let $x_0,\ldots,x_n$ be a set of coordinates on $V$ and $\xi_0,\ldots,\xi_n$ be a set of coordinates on $V^{*h}$.
Notice that a vector field (here we use  vector fields which have non-zero coefficients only at a certain value of $h$)
\begin{equation} \label{X-vf-tangent}
 X= \sum_{j\in \{0,h\}} \sum_{\nu=0}^n a_\nu(X) D_{x_\nu}^{(j)} +
 \sum_{j\in \{0,h\}}
\sum_{\nu=0}^n b_\nu(X) D_{\xi_\nu}^{(j)}
\end{equation}
by eq. (\ref{twisted-act-on-diff-forms}) acts on differential forms in terms of the derivation $i_X$ by the rule:
\begin{align*}
i_X
\left(
d^{(h)}x_i^{q}
\right)
&=a_i(X)^{p^h}, & i_X(d^{(0)}\xi_i)&= b_i(X), 
\\
i_X
\left(
d^{(0)}x_i
\right)
&=a_i(X), & i_X(d^{(h)}\xi_i^q)&= b_i(X)^{p^h}.
\end{align*}

Consider $M\subset \mathbb{P}(V)$ a projective variety and consider the cone $M' \subset V$ seen as an affine variety in $V$. 
Assume that the homogeneous ideal of $M'$ is generated by the homogeneous polynomials $f_1,\ldots, f_r$, and the the set of $h$-non-singular points of $M$ is non-empty. Consider the $n+1$-upple 
\[
\nabla^{(h)}f_i = \left(
  D_0^{(h)} f_i(P),
  D_1^{(h)}  f_i(P), 
\ldots, 
  D_n^{(h)} f_i(P)
\right).
\]
Each $f_i$ defines an $h$-linear form given by 
\begin{equation}
\label{Lih}
L_i^{(h)}:=\sum_{\nu=0}^n   D_\nu^{(h)} f_i(P) x_\nu^{p^h}.
\end{equation}

The $h$-tangent space at $P$
is the variety defined  by the equations $L_i^{(h)}=0$.
The $h$-conormal space is defined as the subset of $V\times V^{*h}$ 
\[
\mathrm{Con}^{(h)}(M)=
\overline{
\left\{ (P,H): P\in M^h_{\mathrm{sm}}, H \text{ is a } p^h-\text{linear form which vanishes on } T_P^{(h)}M
\right\}.
}
\]
It is evident that the $h$-conormal space can be identified to the space of
$p^h$-linear forms on the $h$-normal space $N^{(h)}_M$ defined as 
\[
N^{(h)}_M(P) =T_P^{(h)}V/T_P^{(h)}(M)\cong 
V/T_P^{(h)}M.
\]
Also by the definition of $T_P^{(h)}M$ the fibre of the  $h$-conormal space at the point $P$ for a projective variety defined by the elements 
$f_1,\ldots, f_r$ is the vector subspace of $V^{*h}$ spanned by $L_i^{(h)}$ given in eq. (\ref{Lih}):
\[
\mathrm{Con}^{(h)}_P(M)=\langle L_i^{(h)}: 1\leq i \leq r\rangle_k.
\]
%
%
%
\subsection{The symplectic structure on $V\times V^{*h}$}
\begin{definition} \label{def:Lagrangian}
Let $x_i,\xi_i$ be  coordinates on the vector spaces $V,V^{*h}$ respectively.

A subvariety $\Lambda$ of $V\times V^{*h}$ with non empty $h$-non-singular locus will be  called conical 
$h$-Lagrangian if 
\begin{enumerate}
\item The form 
$\omega=\sum_{j=0}^n d^{(h)}x_j^q \wedge d\xi_j +\sum_{j=0}^n d^{(h)}\xi_j^q \wedge dx_j$
is zero on $\Lambda$.
\item $\dim \Lambda=n$
\item If $(P,H)=(x_0,\ldots,x_n,\xi_0,\ldots,\xi_r)\in \Lambda$ then $(\mu P,\lambda H)=(\mu x_0,\ldots,\mu x_r,\lambda \xi_0,\ldots,\lambda \xi_n) \in \Lambda$ for every, $\mu,\lambda \in k^*$.
\end{enumerate}
\end{definition}

Notice that if 
\[
	X=\sum_{i \in \{0,h\}}\sum_{\nu=0}^n a_{i,\nu}(X) D_{x_\nu}^{(i)} +\sum_{i \in \{0,h\}}\sum_{\nu=0}^n \beta_{i,\nu}(X)D_{\xi_\nu}^{(i)},
\]
then 
\[
\omega(X,Y) :=   i_Yi_X \omega =
\]
\begin{equation}
\label{omega-com}
 = 
\sum_{\nu=0}^n 
\left(
a_{h,\nu}(X)^{p^h} b_{0,\nu}(Y)- a_{h,\nu}(Y)^{p^h}
b_{0,\nu}(X)+
a_{0,\nu}(X) b_{h,\nu}(Y)^{p^h}-
 a_{0,\nu}(Y)b_{h,\nu}(X)^{p^h}
\right).
\end{equation}
If one restricts on $(h,0)$-tangent vectors, i.e. $a_{0,i}(X)=a_{0,i}(Y)=b_{h,i}(X)=b_{h,i}(Y)=0$ for all $i$, then the above computation is compatible with the definition given in section \ref{q-symplecti-forms} since in this case 
\[
\omega(X,Y):=
\sum_{\nu=0}^n 
\left(
a_{h,\nu}(X)^{p^h} b_{0,\nu}(Y)- a_{h,\nu}(Y)^{p^h}b_{0,\nu}(X)
\right).
\]

\begin{lemma} \label{proveLagrangian}
 Assume that $h$ is selected such that
  $\pi_2:\mathrm{Con}^{(h)}M \rightarrow \mathrm{Im}\pi_2=Z$ is seperable.
 If $M^h_{\mathrm{sm}}\neq \emptyset$, then 
the conormal bundle $\mathrm{Con}^{(h)}(M)$ is a Lagrangian manifold of $V\times V^{*h}$. 
\end{lemma}
\begin{proof}
Assume that $M$ is the zero locus of the homogeneous polynomials $F_1,\ldots,F_r$.
When we restrict ourselves to $\mathrm{Con}^{(h)}(M)$ we have that 
\[
\xi_j=\sum_{i=1}^r \lambda_i \left( \left. D_j^{(h)} \right|_P F_i  \right)
\qquad \lambda_i\in k,
\]
and 
\[
d\xi_j=  \sum_{i=1}^r \lambda_i  d\left. D_j^{(h)} \right|_P F_i =
\sum_{i=1}^r \lambda_i  \sum_{\nu=0}^{n}  
\left. D_\nu^{(0)} \right|_P \left. D_j^{(h)} \right|_P F_i dx_\nu.
\]
This means that the first summand of $\omega$ restricted to $\mathrm{Con}^{(h)}(X)$ has the form 
\[
\sum_{j=0}^n d^{(h)}x_j^q \wedge d\xi_j=\sum_{i=1}^r \lambda_i  \sum_{j=0}^n    \sum_{\nu=0}^{n} 
\left. D_\nu^{(0)} \right|_P \left. D_j^{(h)} \right|_P F_i  
\;\;
d^{ (h)}x_j^q\wedge dx_\nu.
\]
In a similar way we have, using eq. (\ref{qraise})
\begin{align*}
\sum_{j=0}^{n} dx_j \wedge d^{(h)}\xi_j^q
&=
\sum_{j=0}^n 
dx_j\wedge 
\left(
\sum_{i=1}^{r} \lambda_i 
\sum_{\nu=0}^{n} 
\left. D_\nu^{(0)} \right|_P \left. D_j^{(h)} \right|_P F_i  
\right)^q
d^{(h)}x_\nu^q \\
&=
\sum_{j=0}^n 
\sum_{i=1}^{r} \lambda_i 
\sum_{\nu=0}^{n} 
\left. D_\nu^{(0)} \right|_P \left. D_j^{(h)} \right|_P F_i 
\;\;
dx_j \wedge d^{(h)}x_\nu^q
\end{align*}
Therefore the form 
\[
\omega=\sum_{j=0}^n d^{(h)}x_j^q \wedge d\xi_j +\sum_{j=0}^n d^{(h)}\xi_j^q \wedge dx_j
\]
 is zero on $\Lambda$.


We now compute the dimension of $\mathrm{Con}^{(h)}(M)$.
If $P$ is an $h$-non-singular point, then  the dimension of the $h$-tangent space equals $\dim M$, therefore the dimension of the conormal space is $n-r$ and the dimension of $\mathrm{Con}^{(h)}(M)=\dim(M)+n-r=n$. 

Finally, if $(x_0,\ldots,x_n,\xi_0,\ldots,\xi_n) \in \mathrm{Con}^{(h)}(M)$ then it is obvious that for $\mu,\lambda\in k^*$ the element $(\mu   x_0,\ldots,\mu x_n,\lambda \xi_0,\ldots,\lambda\xi_n)$ is  
 an element of $\mathrm{Con}^{(h)}(M)$ as well. 
\end{proof}

 \begin{definition}
 A map $f:X \rightarrow Y$ between varieties  will be called generically smooth if the induced map $f_*: T^{(0)}_PX \rightarrow T^{(0)}_{f(P)}Y$ is  surjective  for an open dense subset $U\subset X$. 

 Similarly we will call a map 
 $f:X \rightarrow Y$ $h$-generically smooth if the induced map $f_*: T^{(h)}_PX \rightarrow T^{(h)}_{f(P)}Y$ is  surjective  for an open dense subset $U\subset X$ such that $f(U)$ is an open dense subset of $Y$.
 \end{definition}
 \begin{remark}
 Proposition \ref{p-power-diff} implies that if $f$ is generically smooth then it is $h$-generically smooth. Also if the function field extension $k(X)/k(Y)$ is separable, then there is an open set $U$ so that $f_*$ is smooth for all points in $U$, \cite[p. 169]{Kl:86}, \cite[p.68]{MR2894984}.

 \end{remark}


\begin{remark}
We consider the identification $F:V \rightarrow (V^{*h})^{*h}$ given in theorem \ref{h-duality}. Define the map $\Psi: V\times V^{*h}  \rightarrow  
V^{*h} \times (V^{*h})^{*h} 
$ given by sending $\Psi:(v,w) \mapsto (w,F(v))$. Notice that if $\bar{c},\bar{b}$ are the coordinates of $V^{*h}$, $(V^{*h})^{*h}$, then the
coordinates in $V\times V^{*h}$ are given by $(\bar{b},\bar{c})$, see also the diagram in eq. (\ref{PSI-map}). 
\end{remark}
The following is essential for proving reflexivity. 
\begin{proposition}
\label{Lang-char}
Let $\pi_2:   V\times V^{*h} \rightarrow V^{*h}$ be the second  projection. 
A conical Lagrangian variety  $\Lambda \subset V\times V^{*h} $ 
has projection $Z=\pi_2(\Lambda) \subset V^{*h}$. 
If the set of $h$-smooth points of $Z$ forms an non-empty dense open subset of $Z$ and the map $\pi_2:\Lambda \rightarrow Z \subset V^{*h}$ is 
  $h$-generically smooth, then 
the conormal variety $\mathrm{Con}^{(h)}(Z) \subset V^{*h} \times V$ coincides with $\Lambda$ under the natural identification $\Phi:V\times V^{*h}\rightarrow V^{*h} \times V$.
\end{proposition}
\begin{proof}
The set of $h$-smooth points of $\Lambda$ is non empty by definition \ref{def:Lagrangian}, so it is an open dense set of the irreducible variety $\Lambda$. 
The projection $\pi_2(\Lambda)=Z$ is irreducible since $\Lambda$ is irreducible. 
By assumption the map 
$\pi_2:\Lambda \rightarrow Z$ is 
generically smooth, so we can find an
  open dense set $\Lambda_0\subset \Lambda$, consisted of $h$-smooth points with the additional property that $\pi_2(\Lambda_0)=Z_0$ consists also of $h$-smooth points and moreover  the induced map  $\pi_{2,*}$ forms a surjective  map from $T^{(h)}_P\Lambda \rightarrow T^{(h)}_{\pi_2(P)}Z$.
 In selecting $\Lambda_0$ and $Z_0$ it is essential that open non-empty sets in irreducible varieties are dense. We have the following diagram:
\[
\xymatrix{
  V \times V^{*h} \ar[d]^{\pi_2} \ar@/^2.5pc/[rrrr]^{\cong \;\; \Psi}
  &   \Lambda \ar[d]^{\pi_2} \ar@{_{(}->}[l]
  \ar@/^1.5pc/[rr]_{\cong}
  & \Lambda_0 \ar@{_{(}->}[l] \ar[d]^{\pi_2} 
& \mathrm{Con}^{(h)}(Z) \ar@{^{(}->}[r] \ar[d]& V^{*h}\times V
  \\
  V^{*h} & Z \ar@{_{(}->}[l] & 
   Z_0\ar@{_{(}->}[l] \ar@{=}[r] & Z_0 & & 
}
\]
The map $\Psi$ is the map sending $(x,y) \in V\times V^{*h}$ to $(y,x)\in V^{*h}\times V$. 

We will prove first that for any $h$-smooth point $P\in Z_0\neq \emptyset$,
$\Psi(\pi_2^{-1}(P) \cap \Lambda) \subset \mathrm{Con}^{(h)}(Z)$. 
Let $\xi\in \pi_2^{-1}(P) \cap \Lambda$ having coordinates 
$
\xi=
(\bar{c},\bar{b}) \in V \times V^{*h}.
$
Since  $V$ is a vector space we can see
$\bar{c}$ as an element in
$T^{(0)}V$, i.e.
\[
	\bar{c}=\sum_{i=0}^n c_i D^{(0)}_{x_i}.
\]
We will identify $V^{*h}=V^{*h} \times \{0\}$ with the zero section 
of the bundle (using also theorem \ref{h-duality})
\[
T^{(*h)}(V^{*h})=V^{*h} \times (V^{*h})^{*h}\cong V^{*h} \times V \rightarrow V^{*h}.
\]
Since $\pi_2$ is $h$-generically smooth we can see every element 
 $v\in T^{(h)}_PZ$ written as 
\[
	v=\sum_{i=0}^n a_i D_{\xi_i}^{(h)},
\] 
as an element in $V^{h*}$ under the identification of $V^{*h}$ by its $h$-tangent space. 
Since $\omega$ is zero on $\Lambda$ 
we have by eq. (\ref{omega-com}) for the vector fields
\begin{align*}
\xi &= 
\sum_{i=0}^n c_i \cdot D_{x_i}^{(0)} 
+ \sum_{i=0}^n 0 \cdot D_{x_i}^{(h)}
+ \sum_{i=0}^n b_i\cdot  D_{\xi_i}^{(0)}
+ \sum_{i=0}^n 0 D_{\xi_i}^{(h)}
\\
v &= \sum_{i=0}^n 0 \cdot D_{x_i}^{(0)} 
+ \sum_{i=0}^n 0 \cdot D_{x_i}^{(h)}
+ \sum_{i=0}^n 0\cdot  D_{\xi_i}^{(0)}
+ \sum_{i=0}^n a_i D_{\xi_i}^{(h)}
\end{align*}
\[
	0=\omega(\xi,v)=\sum_{i=0}^{n}a_i^{q}c_i, 
\]
that is the $q$-linear form 
\[
(x_0,\ldots,x_n) \mapsto \sum_{i=0}^n x_i^q c_i,
\]
 vanishes on the element of the tangent space with  coordinates $\bar{a}$. 
The last equation implies that we can see $\Psi(\xi)=(\bar{b}(P),\bar{c}(P))\in V^{*h}\times V$ as an element in $\mathrm{Con}^{(h)}Z_0 \subset V^{*h}\times V$, so
$\Psi(\pi^{-1}(P) \cap \Lambda) \subset \mathrm{Con}^{(h)}(Z)$, notice that the coordinates of $\bar{P}$ satisfy by definition the defining equations of $Z$. 

 So we have $\Psi(\pi^{-1}(P) \cap \Lambda_0) \subset \mathrm{Con}^{(h)}(Z)$ and $\Psi(\Lambda_0)$ is a dense subset of the same dimmension of the irreducible variety $\mathrm{Con}^{(h)}(Z)$ so $\Psi(\Lambda)=\mathrm{Con}^{(h)}(Z)$.
\end{proof}
\begin{theorem}[Reflexivity]
Let $M\in \mathbb{P}(V)$ be an irreducible, reduced projective variety 
generated by $h$-homogeneous elements, which also has a non-empty $h$-non-singular locus. 
Assume that $Z:=\pi_2(\mathrm{Con}^{(h)}(M))$ has a nonempty open set of $h$-non-singular points and that the map 
 \[
 \pi_2:V\times V^{*h} \supset \mathrm{Con}^{(h)}(M) \rightarrow 
 \pi_2(\mathrm{Con}^{(h)}(M)):=Z \subset V^{*h}
 \] is 
  generically smooth. 
 Then 
\[
\Psi(\mathrm{Con}^{(h)}(M))=\mathrm{Con}^{(h)}(Z) \subset V^{*h} \times (V^{*h})^{*h}=V^{*h}\times V.
\]
\end{theorem}
\begin{proof}
The conormal variety $\mathrm{Con}^{(h)}(M)$ which is originaly defined as a subset of $V \times V^{*h}$ can be also seen through $\Psi$ as a subset of $V^{*h}\times V\cong V^{*h}\times (V^{*h})^{*h}$ and by symmetry it is  still Lagrangian of dimension $n$.

Let us now prove that the map $\pi_2$ is $h$-generically smooth. Let $f:X\rightarrow Y$ be a map and suppose
  that $P$ is a smooth point of $X$ and $f(P)$ is a smooth point of $Y$ and $f_* T_PX \rightarrow T_{f(P)}Y$ is surjective. If $X,Y$ are irreducible then such a point $P$ exists since the space of smooth points is a non-empty open dense subset and $f$ is generically smooth. 

Set $\Lambda=\mathrm{Con}^{(h)}(M)$. The map $\pi_2:V\times V^{*h} \supset \mathrm{Con}^{(h)}(M) \rightarrow \pi_2(M):=Z \subset V^{*h}$ is 
 assumed to be 
  generically smooth 
  and by proposition \ref{p-power-diff} the natural map 
  $T^{(h)}_P \Lambda \rightarrow T^{(h)}_{\pi_2(P)} Z$ is the $p^h$-power of the classical differential $d\pi_2:T^{(0)}_P \Lambda \rightarrow T^{(0)}_{\pi_2(P)} Z$.

There is an open dense set $U$ of $\Lambda$ such that for every $P\in U$
we have $n=\dim \Lambda=T^{(0)}_P \Lambda$  (is a classical non-singular point) and $\dim \mathrm{Im} (d\pi_2)=\dim T_{\pi_2(P)} Z=\dim Z$ (surjective differential and $\pi_2(P)$ is classical non-singular) and moreover that $\pi_2(P)$ is $h$-non-singular point of $Z$. Notice that non-empty sets of irreducible varieties are dense and hence have nonempty intersection. 
But $\dim \ker (d \pi_2(P))= \dim \ker (d \pi_2^{p^h}(P))$ hence we obtain    
  \begin{align*}
n &=\dim \ker (d\pi_2(P)) + \dim \mathrm{Im} (d\pi_2(P)) \\
&=\dim \ker (d\pi_2(P))^{p^h} + \dim \mathrm{Im} (d\pi_2(P))^{p^h}.
  \end{align*}
The assumption that $\pi_2(P)$ is $h$-non-singular gives us that $\dim T^{(h)}_{\pi_2(P)}Z=\dim Z$.
On the other hand
$\dim \mathrm{Im} (d\pi_2(P))^{p^h}=\dim Z$
and since 
 $ \mathrm{Im} (d\pi_2(P)^{p^h}) \subset
T_{\pi_2(P)}^{(h)} Z $ if $\pi_2(P)$ is $h$-non-singular, that is $\dim T_{\pi_2(P)}^{(h)} Z=\dim Z $ we finally have that  $d \pi_2(P)^{p^h}$
is surjective. Since we proved that $\pi_2$ is generically smooth reflexivity  follows by theorem \ref{Lang-char}.




\end{proof}


Let $M\subset \mathbb{P}(V)$ be an irreducible, reduced projective variety. We can form the connical $h$-Lagrangian $\mathrm{Con}^{(h)}M \subset V\times V^{*h}$ which has a nonempty open set of $h$-non-singular points and also form the $h$-dual variety $Z=\pi_2(\Lambda)$, where $\pi_2:V\times V^{*h} \rightarrow V^{*h}$ is the second projection. The set $Z$ is irreducible but determining whether  the set of $h$-non-singular points is non-empty is a subtle problem. Irreducible algebraic sets are known to have open dense sets of classical non-singular points.  For proving a reflexivity theorem we need the set of $h$-non-singular points of $Z$ to be nonempty, hence dense subset of $Z$. When $M$ is a hypersurface we have given conditions in lemma \ref{InvertibleHess} so that $Z$ has non-empty set of $h$-non-singular points. 
The condition of $h$-non-singular points requires a computation of the dimension of the algebraic set. Understanding the dimension of the dual variety $Z$ is a subtle task,  see \cite{MR1247282}, \cite{MR555696}, \cite{MR0354657}, \cite{MR0568897}, \cite[sec. 2.5]{MR2014407}.
Let us treat here the following case
\begin{proposition}
\label{complete-inter}
Let $M$ be  a complete intersection described as the zero locus of $r$ polynomials $F_1,\ldots,F_r$ and $\dim M=n-r$, such that all Hasse derivatives $D_{x_j}^{(h)} F_i$ have degree prime to the characteristic. Then the dual variety is a hypersurface. If moreover all Hasse derivatives $D_{x_j}^{(h)} F_i$
have zero $h$-derivatives for all $i=0,\ldots,n$ and $1\leq j \leq r$ then the dual hypersurface has  non-empty $h$-singular locus.
\end{proposition}
\begin{proof}
In this case we can prove that $Z$ has dimension $n-1$ since the coordinates $(\xi_0,\ldots,\xi_n)$ are given by 
\begin{equation}
\label{xi-complete}
\begin{pmatrix}
\xi_0 \\
\vdots \\
\xi_n
\end{pmatrix}
=
\sum_{i=1}^r \lambda_i 
\begin{pmatrix}
D_{x_0}^{(h)} F_i \\
\vdots \\
D_{x_n}^{(h)} F_i
\end{pmatrix}.
\end{equation} 
We now compute the $(n+1) \times (n+1)$-matrix
\begin{equation}
\label{derivaxi}
\begin{pmatrix}
D_{x_0}^{(0)} \xi_0 & \cdots & D_{x_0}^{(0)} \xi_n
\\
\vdots & & \vdots 
\\
D_{x_n}^{(0)} \xi_0 & \cdots & D_{x_n}^{(0)} \xi_n
\end{pmatrix}=
\sum_{i=1}^r 
\lambda_i 
\begin{pmatrix}
D_{x_0}^{(0)} D_{x_0}^{(h)} F_i & \cdots & D_{x_0}^{(0)} D_{x_n}^{(h)} F_i
\\
\vdots & & \vdots 
\\
D_{x_n}^{(0)} D_{x_0}^{(h)} F_i & \cdots & 
D_{x_n}^{(0)} D_{x_n}^{(h)} F_i
\end{pmatrix}.
\end{equation}
If the elements $F_i$  have at least a derivative
$D_{x_\mu}^{(h)}(F_i)$ which is not zero, and degrees $d_{\mu,i}$ which are prime to $p$, then by lemma \ref{InvertibleHess} we obtain that each matrix summand in the right hand side of eq. (\ref{derivaxi}) is generically invertible. 
Without loss of generality we can assume that for $\lambda_1=1,\lambda_2=\cdots=\lambda_r=0$ the matrix in eq. (\ref{derivaxi}) is invertible (change projective coordinates in the projective space $\mathbb{P}^r$ if not.) In this case the subvariety $Z_\lambda$ of the projection $Z$ cut out by equations $\lambda_2=\cdots=\lambda_r$ is locally isomorphic to our original variety $M$ by using Wallace inverse function construction, which allows as to express $(x_0,\ldots,x_n)$ in terms of $(\xi_0,\ldots,\xi_n)$. The dual variety is then ruled in projective spaces with base $Z_\lambda$ and has dimension equal to 
\[
\dim Z_\lambda + r-1= \dim M+r-1=n-r+r-1=n-1.
\]

This means that $Z$ is a hypersurface defined as the zero locus of the polynomial $G(\xi_0,\ldots,\xi_n)$. 
If $Z$ has empty set of $h$-non-singular points,  then all Hasse derivatives $D^{(h)}_{\xi_\nu} G=0$ for $0\leq \nu \leq n$ and this means that $G$ has degree smaller than $q$. Using eq. (\ref{xi-complete}) we can write 
$G$ as a function of $x_0,\ldots,x_n$, depending on $\lambda_1,\ldots,\lambda_r$. Then $G$ is zero on $M$ and this means that
 $G\big(
 \xi_0(x_0,\ldots,x_n),\ldots, \xi_n(x_0,\ldots,x_n) \big)$ is in the ideal generated by  $F_1,\ldots,F_r$.  Let us write 
 \begin{equation}
 \label{member}
G\big(
 \xi_0(\bar{x}),\ldots, \xi_n(\bar{x}) \big)=
 \sum_{i=1}^r g_i(\bar{x}) F_i(\bar{x}). 
 \end{equation}
 The chain rule gives us that (recall we assumed that $D_{\xi_\nu}^{(h)}G=0$ for $0\leq \nu \leq n$)
 \[
\left(
D_{x_0}^{(0)}G,\ldots,D_{x_n}^{(0)}G,
D_{x_0}^{(h)}G,\ldots,D_{x_n}^{(h)}G
\right)=
\left(
D_{\xi_0}^{(0)}G,\ldots,D_{\xi_n}^{(0)}G,
0,\ldots,0
\right)
\begin{pmatrix}[c|c]
J & J' \\
\hline
0 & J^{p^h}
\end{pmatrix}
\]
where
\[
J=
\begin{pmatrix}
D_{x_0}^{(0)}\xi_0 & \cdots & D_{x_n}^{(0)}\xi_0
\\
\vdots &  & \vdots 
\\
D_{x_0}^{(0)}\xi_n & \cdots & D_{x_n}^{(0)}\xi_n
\end{pmatrix}
\text{ and }
J'=
\begin{pmatrix}
D_{x_0}^{(h)}\xi_0 & \cdots & D_{x_n}^{(h)}\xi_0
\\
\vdots &  & \vdots 
\\
D_{x_0}^{(h)}\xi_n & \cdots & D_{x_n}^{(h)}\xi_n
\end{pmatrix}
\]
Therefore if $D_{x_i}^{(h)}\xi_j=0$ for all $0\leq i,j \leq n$ then 
\[
(D_{x_0}^{(h)}G,\ldots,D_{x_n}^{(h)}G)=(0,\ldots,0). 
\]
In this case we have by eq. (\ref{member}) 
\[
\left(
D_{x_0}^{(h)}G(P),\ldots,D_{x_n}^{(h)}G(P)
\right)
=
\sum_{i=1}^r g_i(P) 
\left(
D_{x_0}^{(h)}F_i(P),\ldots,D_{x_n}^{(h)}F_i(P)
\right).
\]
But the vectors 
\[
\nabla^{(h)}F_i=
\left(
D_{x_0}^{(h)}F_i(P),\ldots,D_{x_n}^{(h)}F_i(P)
\right)
\]
are linear independent for every point $P$ in the non-empty set $U$ containing all $h$-non-singular points. This means that for all $P\in U$ $g_i(P)=0$ for $1\leq i \leq r$, which in turn implies that $g_i$ are zero polynomials and $G$ is also zero, a contradiction.

\end{proof}

\subsection{Examples}
Consider the complete intersection in $\mathbb{P}^{n}_k$ given by ($\bar{\lambda}=(\lambda_{1},\ldots,\lambda_{n-2})$)
\begin{equation} \label{defineGFC}
C^k(\bar{\lambda}):=\left \{ \begin{array}{rcc}
              x_0^k+x_1^k+x_2^k&=&0\\
              \lambda_1x_0^k+x_1^k+x_3^k&=&0\\
              \vdots\hspace{1cm} &\vdots &\vdots\\
              \lambda_{n-2}x_0^k+x_1^k+x_n^k&=&0\\
             \end{array}\right \}\subset {\mathbb P}^n_k.
\end{equation}
These curves are called ``generalized Fermat curves'', see \cite{HiRuKoPa}.
We consider the matrix of  $\nabla f_i$ written as rows,
\begin{equation} \label{Jac}
\begin{pmatrix}
k x_0^{k-1} & k x_1^{k-1} &  k x_2^{k-1} & 0& \ldots & 0 \\
\lambda_1 k x_0^{k-1} & k x_1^{k-1} & 0 & k x_3^{k-1}& \ldots &  0 \\
 \vdots      & \vdots    & \vdots  & \vdots  & \vdots \\
 \lambda_{n-2 } k x_0^{k-1} & k x_1^{k-1} & 0 & \ldots &  0 &
  kx_{n}^{k-1}
\end{pmatrix}.
\end{equation}
The conormal space is the subspace in $V^*$ of linear forms spanned by the linear forms 
\[
L_i=\sum_{\nu=0}^n D_{x_i}^{(0)} f_i X_i.
\]
Consider an arbitrary element in the span of $L_i$, $\mu_0,\ldots,\mu_{n-2}\in k$:
\begin{equation} \label{y-ord-def}
\begin{pmatrix}
y_0 \\ y_1 \\ y_2 \\ \vdots \\ y_{n} 
\end{pmatrix}
=k
\begin{pmatrix}
\sum_{\nu=0}^{n-2} \mu_\nu \lambda_\nu x_0^{k-1} \\
\sum_{\nu=0}^{n-2} \mu_\nu x_1^{k-1} \\
\mu_0 x_2^{k-1} \\
\vdots
\\
\mu_{n-2} x_n^{k-1}
\end{pmatrix}.
\end{equation}
The ordinary conormal space is given by
\[
\mathrm{Con}(C^k(\bar{\lambda}))=
\left\{ 
(x_0,\ldots,x_n, y_0,\ldots,y_n): 
\begin{array}{l}
\text{ where } x_0,\ldots,x_n \text{ satisfy eq. (\ref{defineGFC})} \\
\text{ 
and $y_0,\ldots, y_n$ \text{eq. (\ref{y-ord-def})}
}
\end{array}   
\right\}. 
\]
The image of the projection $\pi_2$ is a  codimensional $1$ subvariety, hence a hypersurface given by a single polynomial $F(y_0,\ldots,y_n)=0$.
Finding this polynomial $F$ explicitly is a complicated task in this case.
If $p \mid k-1$ it is clear by equation (\ref{y-ord-def}) that $y_i$ are given as polynomials of $x_i^p$ and the map $\pi_2$ can not be separable, hence reflexivity fails. 

Let us study the conormal space of the dual variety $Z=\pi_2(\mathrm{Con}(C^k{\bar{\lambda}}))$. We see equations (\ref{y-ord-def}) as parametric equations with parameters $\mu_0,\ldots,\mu_{n-2}$. In this case we have that the tangent space is generated by the vectors
\[
V_i:=
\left(
\frac{\partial y_i}{\partial \mu_0}, 
\frac{\partial y_i}{\partial \mu_1},
\ldots,
\frac{\partial y_i}{\partial \mu_{n-2}}
\right)
=\left(
\lambda_i x_0^{k-1},x_1^{k-1},0,\ldots,0,x_i^{k-1},0,\ldots,0
\right) \text{ for } 0\leq i \leq n-2, 
\]
which are subject to the additional condition 
\begin{equation} \label{compa-hyp}
\nabla F \bot V_i \text{  i.e. } \langle \nabla F,V_i\rangle=0.
\end{equation} 

In order to study further eq. (\ref{compa-hyp}) we consider the following cases:

$\bullet$ If $(k-1,p)=1$ then we obtain:
\begin{eqnarray} \label{solve-x}
x_0 & = & 
\left( \frac{y_0}{k\sum_{\nu=0}^{n-2} \mu_\nu \lambda_\nu} \right)^{\frac{1}{k-1}} \\
x_1 & =& 
\left(
\frac{y_1}{k \sum_{\nu=0}^{n-2} \mu_\nu}
\right)^{\frac{1}{k-1}} \nonumber \\
x_i & = & \left( \frac{y_i}{k \mu_{i-2}}\right)^{\frac{1}{k-1}} \qquad \text{for } 2 \leq i \leq n-2. \nonumber
\end{eqnarray}
This way we obtain a relative curve $X\rightarrow \mathbb{P}^{n-1}_k$, where $[\mu_0:\cdots:\mu_{n-2}]$ serve as projective coordinates of $\mathbb{P}^{n-1}_k$. The precise equations in terms of algebraic functions are given by:
\[G_i=
\lambda_i 
\left( \frac{y_0}{k\sum_{\nu=0}^{n-2} \mu_\nu \lambda_\nu} \right)^{\frac{k}{k-1}}
+
\left(
\frac{y_1}{k \sum_{\nu=0}^{n-2} \mu_\nu}
\right)^{\frac{k}{k-1}}
+
\left( 
\frac{y_{i+2}}{k \mu_{i}}
\right)^{\frac{k}{k-1}}=0 \qquad \text{ for } 
0 \leq i \leq n-2. 
\]
The polynomial $F$ can be computed by eliminating $\mu_0,\ldots,\mu_{n-2}$ from the system of the $G_i$. 
We compute (over the open set $\mu_0\mu_1\cdots\mu_{n-2}\neq  0$)
\begin{eqnarray*}
\nabla G_i 
&= &
\frac{k}{k-1}
\left(
\lambda_i   
\left( \frac{y_0}{k\sum_{\nu=0}^{n-2} \mu_\nu \lambda_\nu} \right)^{\frac{1}{k-1}}, 
\left(
\frac{y_1}{k \sum_{\nu=0}^{n-2} \mu_\nu}
\right)^{\frac{1}{k-1}},
\ldots,
\left( 
\frac{y_{i+2}}{k \mu_{i}}
\right)^{\frac{1}{k-1}},
\ldots,
0
\right) 
\\
& =& 
\frac{k}{k-1}
(\lambda_i x_0,x_1,0\ldots,0,x_i,0,\ldots,0).
\end{eqnarray*}
Therefore, the compatibility condition given in eq. (\ref{compa-hyp}) can be replaced by the conditions:
\begin{equation} \label{cond-or}
V_i \bot \nabla G_j \text{ i.e.} \langle V_i, \nabla G_j \rangle=0 \text{ for all } 0\leq i,j \leq n-2.
\end{equation}
We can now confirm that the conditions given in (\ref{cond-or}) are equivalent to the original defining equations for our curve. 
It is clear now that the vector $(x_0,x_1,\ldots,x_n)$ is normal to every generator of the tangent space of the dual variety $Z$ hence 
\[
\mathrm{Con}(Z)=\{(y_0,\ldots,y_n,x_0,\ldots,x_n): F(y_0,\ldots,y_n)=0\}
=\mathrm{Con}(C^k(\bar{\lambda})).
\]
In our computation  it was essential that we were able to express  $x_i$ for $0\leq i \leq n-2$ in terms of $y_i$ for $0\leq i \leq n-2$ in equations (\ref{solve-x}). This could not be done if $p \mid k-1$. We now proceed to the extreme case $k-1$ is a power of $p$.

$\bullet$ Assume that $k=q+1$ for $q=p^h$. Then instead of the matrix given in eq. (\ref{Jac}) we consider the matrix of $\nabla^{(h)}f_i$ given as
\begin{equation} \label{Jach}
\begin{pmatrix}
 x_0 &  x_1 &  x_2 & 0& \ldots & 0 \\
\lambda_1  x_0 &  x_1  & 0 &  x_3& \ldots &  0 \\
 \vdots      & \vdots    & \vdots  & \vdots  & \vdots \\
 \lambda_{n-2 }  x_0 &  x_1 & 0 & \ldots &  0 &
  x_{n}
\end{pmatrix}.
\end{equation}
And now 
\begin{equation}
\label{yh-param}
\begin{pmatrix}
y_0^{(h)} \\ y_1^{(h)} \\ y_2^{(h)} \\ \vdots \\ y_{n}^{(h)} 
\end{pmatrix}
=k
\begin{pmatrix}
\sum_{\nu=0}^{n-2} \mu_\nu \lambda_\nu x_0 \\
\sum_{\nu=0}^{n-2} \mu_\nu x_1 \\
\mu_0 x_2 \\
\vdots
\\
\mu_{n-2} x_n
\end{pmatrix}.
\end{equation}
The relations among elements $y_0^{(h)},\ldots,y_n^{(h)}$ are given by:
\[G_i^{(h)}=
\lambda_i 
\left( \frac{y_0^{(h)}}{\sum_{\nu=0}^{n-2} \mu_\nu \lambda_\nu} \right)^{q+1}
+
\left(
\frac{y_1^{(h)}}{\sum_{\nu=0}^{n-2} \mu_\nu}
\right)^{q+1}
+
\left( 
\frac{y_{i+2}^{(h)}}{ \mu_{i}}
\right)^{q+1}=0 \qquad \text{ for } 
0 \leq i \leq n-2. 
\]
The $h$-conormal space is given by
\[
\mathrm{Con}^{(h)}(C^k(\bar{\lambda}))=
\left\{ 
(x_0,\ldots,x_n, y_0^{(h)},\ldots,y_n^{(h)}): \text{ where }
\begin{array}{c}
x_0,\ldots,x_n 
\text{ 
satisfy eq. (\ref{defineGFC}) 
} 
\\
\text{and } y^{(h)}_0,\ldots, y_n^{(h)} \text{eq. (\ref{Jach})}
\end{array} 
\right\}.
\]
The variety $Z^{(h)}=\pi_2(\mathrm{Con}^{(h)}(C^k(\bar{\lambda})))$ is given by a hypersurface $F^{(h)}(y_0^{(h)},\ldots,y_n^{(h)})=0$, which can be computed by eliminating $\mu_0,\ldots,\mu_{n-2}$ from the system of $G_i^{(h)}$. Similarly we can compute
\begin{eqnarray*}
\nabla^{(h)} G_i^{(h)} 
&= &
\left(
\lambda_i   
 \frac{y_0^{(h)}}{k\sum_{\nu=0}^{n-2} \mu_\nu \lambda_\nu},
\frac{y_1^{(h)}}{k \sum_{\nu=0}^{n-2} \mu_\nu},
\ldots,
\frac{y_{i+2}^{(h)}}{k \mu_{i}},
\ldots,
0
\right) 
\\
& =& 
(\lambda_i x_0,x_1,0\ldots,0,x_i,0,\ldots,0).
\end{eqnarray*}
Again we see equations (\ref{yh-param}) as parametric equations with parameters $\mu_0,\ldots,\mu_{n-2}$. The tangent space is generated by the vectors
\[
V_i^{(h)}:=
\left(
\frac{\partial y_i}{\partial \mu_0}, 
\frac{\partial y_i}{\partial \mu_1},
\ldots,
\frac{\partial y_i}{\partial \mu_{n-2}}
\right)
=\left(
\lambda_i x_0^{k-1},x_1^{k-1},0,\ldots,0,x_i^{k-1},0,\ldots,0
\right) \text{ for } 0\leq i \leq n-2, 
\]
which are subject to the additional condition 
\begin{equation} 
V_i^{(h)} \bot \nabla G_j^{(h)} \text{ i.e.} \langle V_i^{(h)}, \nabla G_j^{(h)} \rangle=0 \text{ for all } 0\leq i,j \leq n-2.
\end{equation} 
As in the zero characteristic case the last conditions are equivalent to the defining equations of the curve. 

 \def\cprime{$'$}


\begin{thebibliography}{10}

\bibitem{MR0354657}
{\em Groupes de monodromie en g\'{e}om\'{e}trie alg\'{e}brique. {II}}.
\newblock Lecture Notes in Mathematics, Vol. 340. Springer-Verlag, Berlin-New
  York, 1973.
\newblock S\'{e}minaire de G\'{e}om\'{e}trie Alg\'{e}brique du Bois-Marie
  1967--1969 (SGA 7 II), Dirig\'{e} par P. Deligne et N. Katz.

\bibitem{MR2014407}
Maks~A. Akivis and Vladislav~V. Goldberg.
\newblock {\em Differential geometry of varieties with degenerate {G}auss
  maps}, volume~18 of {\em CMS Books in Mathematics/Ouvrages de
  Math\'{e}matiques de la SMC}.
\newblock Springer-Verlag, New York, 2004.

\bibitem{Magma1997}
Wieb Bosma, John Cannon, and Catherine Playoust.
\newblock The {M}agma algebra system {I}: {T}he user language.
\newblock {\em J. Symbolic Comput.}, 24(3-4):235--265, 1997.
\newblock Computational algebra and number theory (London, 1993).

\bibitem{CrespoDiff}
Teresa Crespo and Zbigniew Hajto.
\newblock {\em Algebraic groups and differential {G}alois theory}, volume 122
  of {\em Graduate Studies in Mathematics}.
\newblock American Mathematical Society, Providence, RI, 2011.

\bibitem{Da_Silva2001-jk}
A~C da~Silva.
\newblock {\em Lectures on Symplectic Geometry}, volume 1764 of {\em Lecture
  Notes in Mathematics}.
\newblock Springer, 1., st ed. 2001. corr., 2nd printing edition, 2001.

\bibitem{Eisenbud:95}
David Eisenbud.
\newblock {\em Commutative algebra}.
\newblock Springer-Verlag, New York, 1995.
\newblock With a view toward algebraic geometry.

\bibitem{fukkaji}
Satoru Fukasawa and Hajime Kaji.
\newblock The separability of the {G}auss map and the reflexivity for a
  projective surface.
\newblock {\em Math. Z.}, 256(4):699--703, 2007.

\bibitem{MR1247282}
I.~M. Gel\cprime~fand and M.~M. Kapranov.
\newblock On the dimension and degree of the projective dual variety: a
  {$q$}-analog of the {K}atz-{K}leiman formula.
\newblock In {\em The {G}el\cprime fand {M}athematical {S}eminars, 1990--1992},
  pages 27--33. Birkh\"{a}user Boston, Boston, MA, 1993.

\bibitem{gelfand2008discriminants}
I.M. Gelfand, M.~Kapranov, and A.~Zelevinsky.
\newblock {\em Discriminants, Resultants, and Multidimensional Determinants}.
\newblock Modern Birkh{\"a}user Classics. Birkh{\"a}user Boston, 2008.

\bibitem{GilleCentralSimpleAlgebras}
Philippe Gille and Tam\'as Szamuely.
\newblock {\em Central simple algebras and {G}alois cohomology}, volume 101 of
  {\em Cambridge Studies in Advanced Mathematics}.
\newblock Cambridge University Press, Cambridge, 2006.

\bibitem{GossBook}
David Goss.
\newblock {\em Basic structures of function field arithmetic}, volume~35 of
  {\em Ergebnisse der Mathematik und ihrer Grenzgebiete (3) [Results in
  Mathematics and Related Areas (3)]}.
\newblock Springer-Verlag, Berlin, 1996.

\bibitem{Hartshorne:77}
Robin Hartshorne.
\newblock {\em Algebraic geometry}.
\newblock Springer-Verlag, New York, 1977.
\newblock Graduate Texts in Mathematics, No. 52.

\bibitem{Hasse1937}
H.~Hasse.
\newblock Noch eine {B}egründung der {T}heorie der höheren
  {D}ifferentialquotienten in einem algebraischen {F}unktionenkörper einer
  {U}nbestimmten. (nach einer brieflichen {M}itteilung von {F}.{K}. {S}chmidt
  in {J}ena).
\newblock {\em Journal für die reine und angewandte {M}athematik},
  177:215--223, 1937.

\bibitem{hefez89}
Abramo Hefez.
\newblock Nonreflexive curves.
\newblock {\em Compositio Math.}, 69(1):3--35, 1989.

\bibitem{hefklei}
Abramo Hefez and Steven~L. Kleiman.
\newblock Notes on the duality of projective varieties.
\newblock In {\em Geometry today ({R}ome, 1984)}, volume~60 of {\em Progr.
  Math.}, pages 143--183. Birkh\"auser Boston, Boston, MA, 1985.

\bibitem{MR2894984}
Haruzo Hida.
\newblock {\em Geometric modular forms and elliptic curves}.
\newblock World Scientific Publishing Co. Pte. Ltd., Hackensack, NJ, second
  edition, 2012.

\bibitem{HiRuKoPa}
Rub\'en~A. Hidalgo, Aristides Kontogeorgis, Maximiliano Leyton-\'Alvarez, and
  Panagiotis Paramantzoglou.
\newblock Automorphisms of generalized {F}ermat curves.
\newblock {\em J. Pure Appl. Algebra}, 221(9):2312--2337, 2017.

\bibitem{MR555696}
Audun Holme.
\newblock On the dual of a smooth variety.
\newblock In {\em Algebraic geometry ({P}roc. {S}ummer {M}eeting, {U}niv.
  {C}openhagen, {C}openhagen, 1978)}, volume 732 of {\em Lecture Notes in
  Math.}, pages 144--156. Springer, Berlin, 1979.

\bibitem{Hosgood2016-sp}
Timothy Hosgood.
\newblock An introduction to varieties in weighted projective space.
\newblock 8~April 2016.

\bibitem{kaji}
Hajime Kaji.
\newblock On the inseparable degrees of the {G}auss map and the projection of
  the conormal variety to the dual of higher order for space curves.
\newblock {\em Math. Ann.}, 292(3):529--532, 1992.

\bibitem{kleipie}
Steven Kleiman and Ragni Piene.
\newblock On the inseparability of the {G}auss map.
\newblock In {\em Enumerative algebraic geometry ({C}openhagen, 1989)}, volume
  123 of {\em Contemp. Math.}, pages 107--129. Amer. Math. Soc., Providence,
  RI, 1991.

\bibitem{MR0568897}
Steven~L. Kleiman.
\newblock The enumerative theory of singularities.
\newblock pages 297--396, 1977.

\bibitem{Kl:86}
Steven~L. Kleiman.
\newblock Tangency and duality.
\newblock In {\em Proceedings of the 1984 Vancouver conference in algebraic
  geometry}, pages 163--225, Providence, R.I., 1986. Amer. Math. Soc.

\bibitem{MilneLEC}
James~S. Milne.
\newblock Lectures on etale cohomology (v2.21), 2013.
\newblock Available at www.jmilne.org/math/.

\bibitem{NeukirchCohoNumFields}
J\"urgen Neukirch, Alexander Schmidt, and Kay Wingberg.
\newblock {\em Cohomology of number fields}, volume 323 of {\em Grundlehren der
  Mathematischen Wissenschaften [Fundamental Principles of Mathematical
  Sciences]}.
\newblock Springer-Verlag, Berlin, second edition, 2008.

\bibitem{Schmidt39}
Friedrich~Karl Schmidt.
\newblock Die {W}ronskische {D}eterminante in beliebigen differenzierbaren
  {F}unktionenk{\"{o}}rpern.
\newblock {\em Math. Z.}, 45(1):62--74.

\bibitem{ShafaBo1}
Igor~R. Shafarevich.
\newblock {\em Basic algebraic geometry. 1}.
\newblock Springer-Verlag, Berlin, second edition, 1994.
\newblock Varieties in projective space, Translated from the 1988 Russian
  edition and with notes by Miles Reid.

\bibitem{ShafaBo2}
I.R. Shafarevich.
\newblock {\em Basic Algebraic Geometry 2}.
\newblock Basic Algebraic Geometry. Springer-Verlag, 1994.

\bibitem{VakilSea}
Ravi Vakil.
\newblock {\em The rising sea}.
\newblock 2017.

\bibitem{wal}
Andrew~H. Wallace.
\newblock Tangency and duality over arbitrary fields.
\newblock {\em Proc. London Math. Soc. (3)}, 6:321--342, 1956.

\end{thebibliography}
\end{document}